\documentclass{amsart}

\usepackage{amsmath,amssymb,amsbsy,amsthm,amsfonts,mathtools,dsfont,hyperref,mathrsfs}

\usepackage{enumitem}
\usepackage{wrapfig}
\usepackage[top=3.5cm,bottom=3.5cm,left=3.5cm,right=3.5cm]{geometry}
\usepackage[normalem]{ulem}

\usepackage{graphicx}
\usepackage{tikz}
\usepackage{tikzscale}
\usetikzlibrary{intersections,decorations.markings}
\tikzset{>=latex}
\usepackage{pgfplots}
\pgfplotsset{compat=newest}
\usepgfplotslibrary{fillbetween}

\allowdisplaybreaks[3]

\DeclareMathOperator*{\esssup}{ess\ sup}

\newcommand{\N}{\mathbb{N}}
\newcommand{\R}{\mathbb{R}}

\newcommand{\tr}{\operatorname{Tr}}

\def\gscale{1.4}
\def\graphschody{
  \begin{tikzpicture}[scale=\gscale]
    \footnotesize
    \draw[ultra thin] (-0.2,  0.0) -- (3.2, 0.0) node[below]{$|\nabla u|$};
    \draw[ultra thin] ( 0.0, -0.2) -- (0.0, 3.2) node[left ]{$|\j|$};
    \draw[variable=\t,very thick,every node/.style={inner sep=1,outer sep=1}]
      plot[domain=0:1] ({0},{\t})
    --plot[domain=0:1] ({\t},{1})
    --plot[domain=1:2] ({1},{\t})
    --plot[domain=1:2] ({\t},{2})
    --plot[domain=2:3] ({2},{\t})
    --plot[domain=2:3] ({\t},{3})
    ;
  \end{tikzpicture}
}
\def\sscale{1.4}
\def\graphsmykala{
  \begin{tikzpicture}[scale=\sscale]
    \footnotesize
    \draw[ultra thin] (-0.2,  0.0) -- (3.2, 0.0) node[below]{$|\nabla u|$};
    \draw[ultra thin] ( 0.0, -0.2) -- (0.0, 3.2) node[left ]{$|\j|$};
    \draw[variable=\t,very thick,every node/.style={inner sep=1,outer sep=1}]
      plot[domain=0:1] ({0},{\t})
    --plot[domain=0:1] ({\t},{1})
    --plot[domain=1:3] ({\t},{\t})
    ;
  \end{tikzpicture}
}
\def\grafg{
\draw[variable=\t,every node/.style={inner sep=1,outer sep=1}]
      plot[domain=0:1] ({0},{\t})
    --plot[domain=0:1] ({\t},{1})
    --plot[domain=1:2] ({1},{\t})
    --plot[domain=1:2] ({\t},{2})
    --plot[domain=2:3] ({2},{\t})
    --plot[domain=2:3] ({\t},{3})
    ;
}

\def\grafAee#1#2{
\def\ee{#1}
\def\gAee{(1+\ee^2)/\ee*\t}
\def\ggAee{\ee*\t+1}
    \draw[variable=\t,#2,every node/.style={inner sep=1,outer sep=1}]
      plot[domain=0:{\ee}]             ({\t},{\gAee})
    --plot[domain={\ee}:{1+\ee}]       ({\t},{\ggAee})
    --plot[domain={1+\ee}:{1+2*\ee}]   ({\t},{\gAee-1/\ee})
    --plot[domain={1+2*\ee}:{2+2*\ee)}]({\t},{\ggAee+1})
    --plot[domain={2+2*\ee)}:{2+3*\ee}]({\t},{\gAee-2/\ee})
    --plot[domain={2+3*\ee}:{3+3*\ee}] ({\t},{\ggAee+2})
    ;
}

\def\1scale{1.5}

\def\grafge#1#2{
\def\ee{#1}
\def\gge{\t/\ee}
\def\ggge{\ee*\t+1}
    \draw[variable=\t,{#2},every node/.style={inner sep=1,outer sep=1}]
      plot[domain=0:{\ee/(1-\ee^2)}]                     ({\t},{\gge})
    --plot[domain={\ee/(1-\ee^2)}:{1/(1-\ee)}]           ({\t},{\ggge})
    --plot[domain={1/(1-\ee)}:{1/(1-\ee)+\ee/(1-\ee^2)}] ({\t},{\gge-1/\ee})
    --plot[domain={1/(1-\ee)+\ee/(1-\ee^2)}:{2/(1-\ee)}] ({\t},{\ggge+1})
    --plot[domain={2/(1-\ee)}:{2/(1-\ee)+\ee/(1-\ee^2)}] ({\t},{\gge-2/\ee})
    --plot[domain={2/(1-\ee)+\ee/(1-\ee^2)}:{3/(1-\ee)}] ({\t},{\ggge+2})
    ;
}

\def\2scale{1.5}

\def\grafte#1#2{
\def\ee{#1}
\def\gte{\t*(\ee+2)/\ee}
\def\ggte{\t*\ee/(2+\ee)}
    \draw[variable=\t,{#2},every node/.style={inner sep=1,outer sep=1}]
      plot[domain=0:{\ee/(2+2*\ee)}]       ({\t},{\gte})
    --plot[domain={\ee/(2+2*\ee)}:{1}]     ({\t},{\ggte+2/(2+\ee)})
    --plot[domain={1}:{\ee/(2+2*\ee)+1}]   ({\t},{\gte-2/\ee})
    --plot[domain={\ee/(2+2*\ee)+1}:{2}]   ({\t},{\ggte+4/(2+\ee)})
    --plot[domain={2}:{\ee/(2+2*\ee)+2}]   ({\t},{\gte-4/\ee})
    --plot[domain={\ee/(2+2*\ee)+2}:{3}]   ({\t},{\ggte+6/(2+\ee)})
    ;
}

\def\3scale{1.5}

\def\porovnanie#1#2{
  \begin{tikzpicture}[scale=1.35]
    \footnotesize
\def\epsi{#1}
\def\vel{#2}
    \draw[ultra thin] (-0.2,  0.0) -- (\vel+0.2, 0.0) node[below]{$|\dd|$};
    \draw[ultra thin] ( 0.0, -0.2) -- (0.0, \vel+0.2) node[left ]{$|\j|$};
    \grafg
    \grafAee{\epsi}{dash dot}
    \grafge{\epsi}{dotted}
    \grafte{\epsi}{dashed}
  \end{tikzpicture}
}

\begin{document}

\newcommand{\ddd}{\,{\rm d}}

\def\note#1{\marginpar{\small #1}}
\def\tens#1{\pmb{\mathsf{#1}}}
\def\vec#1{\boldsymbol{#1}}
\def\norm#1{\left|\!\left| #1 \right|\!\right|}
\def\fnorm#1{|\!| #1 |\!|}
\def\abs#1{\left| #1 \right|}
\def\ti{\text{I}}
\def\tii{\text{I\!I}}
\def\tiii{\text{I\!I\!I}}

\newcommand{\loc}{{\rm loc}}
\def\diver{\mathop{\mathrm{div}}\nolimits}
\def\grad{\mathop{\mathrm{grad}}\nolimits}
\def\Div{\mathop{\mathrm{Div}}\nolimits}
\def\Grad{\mathop{\mathrm{Grad}}\nolimits}
\def\cof{\mathop{\mathrm{cof}}\nolimits}
\def\det{\mathop{\mathrm{det}}\nolimits}
\def\lin{\mathop{\mathrm{span}}\nolimits}
\def\pr{\noindent \textbf{Proof: }}

\def\pp#1#2{\frac{\partial #1}{\partial #2}}
\def\dd#1#2{\frac{\d #1}{\d #2}}
\def\vec#1{\boldsymbol{#1}}

\def\0{\vec{0}}
\def\A{\mathcal{A}}
\def\B{\mathcal{B}}
\def\b{\vec{b}}
\def\C{\mathcal{C}}
\def\c{\vec{c}}
\def\D{\nabla \u}
\def\DD{\vec{D}}
\def\BB{\vec{B}}
\def\e{\varepsilon}
\def\f{\vec{f}}
\def\F{\vec{F}}
\def\tF{\tilde{\F}}
\def\g{\vec{g}}
\def\G{\vec{G}}
\def\cG{\mathcal{\G}}
\def\I{\vec{I}}
\def\Im{\text{Im}}
\def\j{\vec{j}}
\def\dd{\vec{d}}
\def\k{\vec{k}}
\def\n{\vec{n}}
\def\q{\vec{q}}
\def\S{\vec{J}}
\def\s{\vec{s}}
\def\T{\vec{T}}
\def\u{\vec{u}}
\def\vp{\vec{\varphi}}
\def\vv{\vec{v}}
\def\w{\vec{w}}
\def\W{\vec{W}}
\def\x{\vec{x}}
\def\z{\vec{Z}}
\def\tz{\tilde{\z}}
\def\Z{\vec{Z}}
\def\X{\vec{X}}
\def\Y{\vec{Y}}
\def\balfa{\vec{\alpha}}

\def\Ae{\A_\e}
\def\Aee{\Ae^\e}
\def\Aeetilde{\tilde{A}_\e^\e}
\def\Be{\B_\e}
\def\Bee{\Be^\e}
\def\De{\DD^\e}
\def\Dve{\D^\e}
\def\oD{\overline{\DD}}
\def\tD{\tilde{\DD}}
\def\Dn{\DD^\e}
\def\Dno{\overline{\Dn}}
\def\Dnt{\tilde{\Dn}}
\def\Dm{\DD^\eta}
\def\Dmo{\overline{\Dm}}
\def\Dmt{\tilde{\Dm}}
\def\Se{\S^\e}
\def\se{\s^\e}
\def\ose{\overline{\se}}
\def\oS{\overline{\S}}
\def\tS{\tilde{\S}}
\def\Sn{\S^\e}
\def\Sno{\overline{\Sn}}
\def\Snt{\tilde{\Sn}}
\def\Sm{\S^\eta}
\def\Smo{\overline{\Sm}}
\def\Smt{\tilde{\Sm}}
\def\ve{\u^\e}
\def\ove{\overline{\ve}}
\def\vd{\u^\delta}
\def\sd{\s^\delta}
\def\Sd{\S^\delta}
\def\Dd{\D^\delta}

\def\Wnd#1{W^{1,#1}_{\n, \diver}}
\def\Wndr{W^{1,r}_{\n, \diver}}

\def\o{\Omega}
\def\po{\partial \Omega}
\def\dt{\frac{\d}{\d t}}
\def\pt{\partial_t}
\def\ig{\int_{\Gamma} \!}
\def\igt{\int_{\Gamma_t} \!}
\def\io{\!\int_{\Omega} \!}
\def\ipo{\!\int_{\partial \Omega} \!}
\def\iq{\int_{Q} \!}
\def\iqt{\int_{Q_t} \!}
\def\it{\int_0^t \!}
\def\iT{\int_0^T \!}
\def\d{\, \textrm{d}}

\def\mn{\mathcal{P}}
\def\du{\mathcal{W}}
\def\tr{\text{tr}}
\def\tow{\rightharpoonup}


\newtheorem{theorem}{Theorem}[section]
\newtheorem{lemma}[theorem]{Lemma}
\newtheorem{proposition}[theorem]{Proposition}
\newtheorem{remark}[theorem]{Remark}
\newtheorem{corollary}[theorem]{Corollary}
\newtheorem{definition}[theorem]{Definition}
\newtheorem{example}[theorem]{Example}

\numberwithin{equation}{section}

\title[Nonlinear parabolic problems with implicit constitutive
equations]{On nonlinear problems of parabolic type with implicit constitutive
equations involving flux}

\thanks{M.~Bul\'i\v{c}ek and J.~M\'alek acknowledge the support of the project No. 18-12719S financed by the Czech Science foundation (GA\v{C}R). E.~Maringov\'a acknowledges support from Charles University Research program UNCE/SCI/023, the grant SVV-2020-260583
 by the Ministry of Education, Youth and Sports, Czech Republic and from the Austrian Science Fund (FWF), grants P30000, W1245, and F65.  M.~Bul\'i\v{c}ek and J.~M\'alek are members of the {Ne\v{c}as} Center for Mathematical Modelling.}

\author[M.~Bul\'{i}\v{c}ek]{Miroslav Bul\'i\v{c}ek}
\address{Charles University, Faculty of Mathematics and Physics, Mathematical Institute, Sokolovsk\'{a} 83, 186~75, Prague, Czech Republic}
\email{mbul8060@karlin.mff.cuni.cz}

\author[J. M\'alek]{Josef M\'alek}
\address{Charles University, Faculty of Mathematics and Physics, Mathematical Institute, Sokolovsk\'{a} 83, 186~75, Prague, Czech Republic}
\email{malek@karlin.mff.cuni.cz}

\author[E.~Maringov\'{a}]{Erika Maringov\'{a}}
\address{Institute of Science and Technology Austria, Am Campus 1, 3400 Klosterneuburg, Austria}
\email{erika.maringova@ist.ac.at}


\begin{abstract}
We study systems of nonlinear partial differential equations of parabolic type, in which the elliptic operator is replaced by the first order divergence operator acting on a flux function, which is related to the spatial gradient of the unknown through an additional implicit equation. This setting, broad enough in terms of applications, significantly expands the paradigm of nonlinear parabolic problems. Formulating four conditions concerning the form of the implicit equation, we first show that these conditions describe a  maximal monotone $p$-coercive graph. We then establish the global-in-time and large-data existence of a (weak) solution and its uniqueness. To this end, we adopt and significantly generalize Minty's method of monotone mappings. A unified theory, containing several novel tools, is developed in a way to be tractable from the point of view of numerical approximations.
\end{abstract}

\maketitle

\section{Introduction}\label{intro}

An initial-boundary-value problem for a scalar linear parabolic equation is usually formulated in the following way:
\begin{equation}
\begin{aligned}
\text{For any given  }\o \subset \R^d,~&T>0,~u_0: \o \to \R,~u_D: \Sigma_D \to \R,\\ f: Q \to \R,~g: \Sigma_N \to \R, &\text{ find a function } u:Q \to \R \text{ satisfying\footnotemark} \\
&\begin{aligned}\label{1P}
\pt u - \Delta u &=f &&\text{in } Q , \\
u &= u_D &&\text{on } \Sigma_D, \\
\nabla u \cdot \n &=g &&\text{on } \Sigma_N, \\
 u (0,\cdot) &=u_0 &&\text{in } \o.
\end{aligned}
\end{aligned}
\end{equation}
\footnotetext{Instead of $\Delta u$, we could consider a general linear elliptic operator of the form $\diver (\mathbb{A}(t, \x)\nabla u)$, where $\mathbb{A}$ fulfills the ellipticity condition: there exists $\alpha >0$ such that for all $\vec{z}\in \R^d$, $\vec{z} \cdot \mathbb{A}(t,\x) \vec{z} \geq \alpha |\vec{z}|^2$. Although this extension has some positive aspects, we avoid it so as to keep the introductory section simpler from the point of view of notation.}%
Here, $\o \subset \R^d$, $d\geq 2$, is supposed to be an open, bounded, connected set with a Lipschitz boundary $\po$ consisting of two mutually disjoint parts $\Gamma_D$ and $\Gamma_N$ so that $\overline{\Gamma_D \cup \Gamma_N}=\po$. Furthermore, $\n: \po \to \R^d$ is the outer unit normal, $Q:= (0,T)\times \o$, $\Sigma_D:= (0,T)\times \Gamma_D$, and $\Sigma_N:= (0,T)\times \Gamma_N$. The set $\o$ having the above properties will be called \underline{Lipschitz domain} in what follows.

The problem~\eqref{1P} can be equivalently rewritten in the form:
\begin{equation}
\label{2}
\begin{aligned}
\text{Find, for any data given in }&\text{\eqref{1P}, a couple of functions } (u, \j): Q \to \R \times \R^d \text{ satisfying}\\
&\begin{aligned}
\pt u - \diver \j &=f &&\text{in } Q, \\
\j &=\nabla u &&\text{in } Q, \\
u &= u_D &&\text{on } \Sigma_D, \\
\j \cdot \n &=g &&\text{on } \Sigma_N, \\
 u (0,\cdot) &=u_0 &&\text{in } \o.
\end{aligned}
\end{aligned}
\end{equation}
The mixed formulation~\eqref{2} has several advantages: it frequently reflects how the problem is generated (as the first equation in~\eqref{2} is in the form of a balance equation  and the second equation is the simplest example of a \emph{constitutive} equation describing how the flux~$\j$ and~$\nabla u$ are related); it is focused simultaneously on the quantities of interest, i.e. on~$u$ and the flux~$\j$; and it also serves as the starting point for the construction of certain numerical methods that are different from those designed for~\eqref{1P}.

The a~priori information associated with~\eqref{2}, the so-called energy (in)equality, provides a natural functional setting in which a robust mathematical theory should be developed. Taking for simplicity~$u_D=0$, the energy equality associated with~\eqref{2} takes the form
\begin{equation*}
\frac{1}{2}\|u(t)\|_2^2 + \it \io \j \cdot \nabla u \d x \d {\tau} = \it \io f u \d x \d {\tau} + \it \int_{\Gamma_N} g u \d S \d {\tau} + \frac{1}{2}\|u_0\|_2^2.
\end{equation*}
Using $\j = \nabla u$ twice, one observes that $\j \cdot \nabla u = \frac{|\j|^2}{2} + \frac{|\nabla u|^2}{2}$,  hence the energy equality leads to
\begin{equation}\label{3}
\|u(t)\|_2^2 + \it \|\j\|_2^2 \d {\tau} + \it \|\nabla u\|_2^2 \d {\tau} = 2 \it \io f u \d x \d {\tau} + 2 \it \int_{\Gamma_N} g u \d S \d {\tau} + \|u_0\|_2^2.
\end{equation}
It is well known that in the setting dictated by~\eqref{3}, the problem~\eqref{1P}, as well as~\eqref{2}, are well-posed. For example, the following theorem can be found in \cite[Theorem~10.1, page 616]{LSU}, see also~\cite[Chapter 2]{MaSt} for the elliptic version:
\begin{theorem}\label{LADY}
Let $\Omega \subset \R^d$ be a Lipschitz domain and $T>0$. Furthermore, assume that $u_0 \in L^2(\Omega)$, $u_D\in W^{1,2}(0,T; (W^{1,2}_{\Gamma_D}(\Omega))^*)\cap L^2(0,T; W^{1,2}(\Omega))$,  $f \in L^{2}(0,T; (W^{1,2}(\Omega))^*)$ and $g\in L^2(0,T; (W^{\frac12,2}(\partial \Omega))^*)$. Then, there exists a unique $u:Q\to\R$ fulfilling\footnote{Here, $W^{1,2}_{\Gamma_D}(\Omega)$ is the standard Sobolev space consisting of functions vanishing on $\Gamma_D$.}
\begin{align*}
u &\in  L^2(0, T; W^{1,2}(\Omega)) \cap  \C([0, T]; L^2(\Omega)), \\
\pt u &\in L^{2}(0, T; (W^{1,2}_{\Gamma_D}(\Omega)){^*}),
\end{align*}
such that $u=u_D$ on $\Sigma_D$ and for a.a. $t\in (0,T)$ there holds:
\begin{subequations}
\begin{equation}\label{WFSra}
\langle \pt u,  \varphi \rangle_{W^{1,2}_{\Gamma_D}(\Omega)} + \io \nabla u \cdot \nabla \varphi \d x  = \langle f,\varphi \rangle_{W^{1,2}_{\Gamma_D}(\Omega)} \quad  \textrm{ for all } \varphi\in W^{1,2}_{\Gamma_D}(\Omega).
\end{equation}
The initial condition is attained in the strong sense,  i.e.,
\begin{equation}\label{ICSra}
\lim_{t\to 0_+} \|u(t) - u_0\|_{L^2(\Omega)} = 0.
\end{equation}
\end{subequations}
\end{theorem}

The aim of this study is to present a robust and possibly elegant mathematical theory for a class of problems similar to~\eqref{2}, with one remarkable difference, namely, the linear relation between the flux $\j$ and $\nabla u$ is replaced by an implicit constitutive equation
\begin{equation}\label{4}
\g(\j,\nabla u) = \0~\text{ in } Q,
\end{equation}
where $\g: \R^d \times \R^d \to \R^d$ is a given continuous (nonlinear) function. The key examples we have in mind, and which will be covered by the theory presented below, are listed in Table~\ref{tab:1}{, where, for~$z \in \R$, we use the notation~$z^+:= \max \{z, 0\}$.}
\renewcommand{\arraystretch}{1.5}
\begin{table}[h]
\label{tab:1}
\begin{tabular}{|l@{\hspace{.2in}}|l@{\hspace{.2in}}|}
\hline
$\j = \k(\nabla u)$ & $\nabla u = \k(\j)$ \\ \hline\hline
$\j = |\nabla u|^{p-2} \nabla u$ & $\nabla u = |\j|^{p'-2}\j$ \\ \hline
$\j =  (1+ |\nabla u|)^{p-2} \nabla u$ &  $\nabla u = (1+ |\j|)^{p'-2}\j$ \\ \hline
$\j = (1+ |\nabla u|^2)^{\frac{p-2}{2}} \nabla u$ &  $\nabla u = (1+ |\j|^2)^{\frac{p'-2}{2}}\j$ \\ \hline
$\j = (|\nabla u|-\delta_*)^+ \frac{\nabla u}{|\nabla u|}$ &  $\nabla u = (|\j|-\sigma_*)^+ \frac{\j}{|\j|}$ \\ \hline
\end{tabular}\vspace{.3cm}
\caption{The implicit relation~\eqref{4} contains two classes of explicit relations (the left and the right column) that include various power-law relations as well as relations with activations (degeneracies). Regarding the parameters, $p \in (1, +\infty)$, $p' = p/(p-1)$, and $\delta_*, \sigma_* >0$. The structure of these relations is motivated by the classification of incompressible fluid models presented in \cite{blechta2020}.}
\end{table}

Examples that belong to the class~\eqref{4}, and are covered by the theory presented below, but cannot be included into either column in Table~\ref{tab:1}, are sketched in Figure~\ref{figure1}. { These are simple examples that together with a class of equations of the form $\alpha(|\j|,|\nabla u|)\j = \beta(|\j|,|\nabla u|)\nabla u$ underlie the full strength of the implicit constitutive theory.}
\begin{figure}[h]
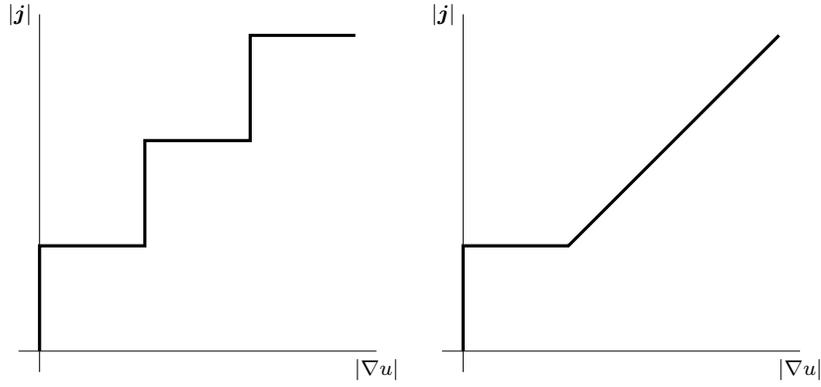

\begin{tabular}{ll}
\graphschody
\graphsmykala
\end{tabular}
\caption{{The relation on the left is a step function from the viewpoint of both $\j$ and $\nabla u$. It can be described as the $\sqrt{2}$-periodic zig-zag function with amplitude $\sqrt{2}/2$ rotated by $45$ degrees in the $(\j,\nabla u)$-plane. The relation on the right depicts one simple step followed by the linear relation $\j=\nabla u$. Both curves are continuous, but neither of them can be written in the form~$\j=\k(\nabla u)$ or~$\nabla u=\k(\j)$.}}
\label{figure1}
\end{figure}



Note that Table~\ref{tab:1} includes relations that lead to the standard $p$-Laplace operator $\diver(|\nabla u|^{p-2} \nabla u)$ and their variants $\diver((1+ |\nabla u|^2)^{\frac{p-2}{2}} \nabla u)$ or $\diver((1+ |\nabla u|)^{p-2} \nabla u)$, however, it also contains less investigated forms, namely
\begin{equation*}
\diver \j~\text{  with  }~\nabla u = (1+ |\j|^2)^{\frac{p'-2}{2}}\j.
\end{equation*}
It also covers degenerate operators of type $\diver((|\nabla u|-\delta_*)^+ \frac{\nabla u}{|\nabla u|})$ as well as the ``multivalued" relations of the type
\begin{equation*}
\diver \j~\text{  with  }~\nabla u = (|\j|-\sigma_*)^+ \frac{\j}{|\j|},
\end{equation*}
which are more frequently written in the form
\begin{align*}
|\j| &\leq \sigma_* ~~~\Longleftrightarrow~~~\nabla u = \0, \\
|\j| &> \sigma_* ~~~\Longleftrightarrow ~~~\j = \sigma_* \frac{\nabla u}{|\nabla u|} + \nabla u.
\end{align*}

In the theory developed in this work we show that the implicit relation~\eqref{4} is fulfilled almost everywhere (i.e. point-wise) in $Q$
 provided that $\g$ fulfills the following conditions:
\begin{itemize}
\item[(g1)] $\g$ is Lipschitz continuous, i.e. $\g \in \mathcal{C}^{0,1}(\R^d \times \R^d)^d$, and $\g(\0,\0)=\0$;
\item[(g2)] for almost all $(\j, \dd) \in \R^d \times \R^d$:
\begin{equation*}
\qquad \quad \g_{\j}(\j,\dd)\ge 0, \quad \g_{\dd}(\j,\dd)\le 0, \quad \g_{\j}(\j,\dd)-\g_{\dd}(\j,\dd)>0 \quad \textrm{ and }\quad \g_{\dd}(\j,\dd)( \g_{\j}(\j,\dd))^T\le 0;
\end{equation*}
\item[(g3)] one of the following conditions holds:
\begin{equation*}
\text{either} \quad\forall \dd\in \mathbb{R}^d \quad \liminf_{|\j|\to +\infty} \g(\j, \dd) \cdot \j > 0 \quad \text{or} \quad \forall \j\in \mathbb{R}^{d}\quad \limsup_{|\dd|\to +\infty} \g(\j, \dd) \cdot \dd< 0;
\end{equation*}
\item[(g4)] for an arbitrary but fixed $p\in (1, \infty)$ there exist $c_1, c_2>0$ such that, for all $(\j,\dd)\in \R^d \times \R^d$ fulfilling~$\g(\j,\dd)=\0$, the following condition holds:
\begin{equation*}
\j \cdot \dd \ge c_1 (|\j|^{p'} + |\dd|^p) -c_2, \qquad p':= {p}/({p-1}).
\end{equation*}
\end{itemize}
In (g1)--(g4), we used the following notation. The mappings $\g_{\j}, \g_{\dd}: \mathbb{R}^d \times \mathbb{R}^d \to \mathbb{R}^{d\times d}$ are defined via
$$
\begin{aligned}
\g_{\j}(\j, \dd):= \frac{\partial \g(\j,\dd)}{\partial \j}\quad \textrm{ and } \quad
 \g_{\dd}(\j, \dd):= \frac{\partial \g(\j,\dd)}{\partial \dd},
\end{aligned}
$$
which, written component-wise, means (here, $\g=(g_1,\ldots, g_d)$, $\j=(j_1,\ldots, j_d)$, and $\dd=(d_1,\ldots, d_d)$)
$$
\begin{aligned}
(\g_{\j})_{ab}:= \frac{\partial g_a(\j,\dd)}{\partial j_b}\quad \textrm{ and } \quad
 (\g_{\dd})_{ab}:= \frac{\partial g_a(\j,\dd)}{\partial d_b}.
\end{aligned}
$$
Further, $(\g_{\dd})^T$ denotes the transpose matrix to $\g_{\dd}$ and  $\g_{\j}(\j,\dd)(\g_{\dd}(\j,\dd))^T$ is the standard matrix multiplication. Also, for any matrix $A\in \mathbb{R}^{d\times d}$, the expression $A\ge 0$ means that for any $\x\in \mathbb{R}^d$ there holds
$$
A\x \cdot \x\ge 0 \qquad \textrm{(which, written in terms of components, is $\sum_{i,j=1}^d A_{ij}x_i x_j \ge 0$)}.
$$
In addition, if we write $A>0$ then we mean that the above inequality is strict for all $\x\neq \0$. Also, since we can replace $\g$ by $-\g$, it is clear that all inequalities in (g2) and (g3) can be equivalently formulated with the opposite sign except the last inequality in (g2).

{ The Lipschitz continuity of $\g$ stated in (g1) is required in order to guarantee the (almost everywhere) existence of derivatives, which are used in (g2). It would be possible to require merely continuity of $\g$ and replace (g2) by the condition:
\begin{itemize}
\item[(g2)$^*$] {for any $(\j_1, \dd_1), (\j_2, \dd_2) \in \R^d \times \R^d$ satisfying~$\g(\j_i,\dd_i)=\0$, $i=1,2$, the following condition holds:}
\begin{equation*}
(\j_1 - \j_2) \cdot (\dd_1-\dd_2)\geq 0.
\end{equation*}
\end{itemize}
However, in the setting of implicit equations of the form \eqref{4}, it seems easier to check (g2) than to prove that (g2)$^*$ holds, unless the considered constitutive equation belongs to one of the explicit classes given in Table~\ref{tab:1}.}

The assumptions (g1)--(g4) are fulfilled by all constitutive equations listed in Table~\ref{tab:1}. As the assumptions (g1)--(g4) might not seem intuitive at the first sight, for the reader's convenience, we show the validity of (g1)--(g4) for a few selected constitutive equations listed above in Appendix~\ref{App1}. Defining $\balfa:=\{(\j,\dd)\in \mathbb{R}^d \times \mathbb{R}^d{;} \g(\j,\dd)=\0\}$, we will also show (see Section~\ref{maxmon}, Lemma~\ref{GvsA} and Lemma~\ref{onto}) that $\balfa$ is a maximal monotone $p$-coercive graph (see Section~\ref{maxmon} for the definitions). More precisely, the assumption (g2) implies that $\balfa$ is monotone and then (g3) in combination with the continuity of $\g$ guarantees that $\balfa$ is maximal monotone. We prefer to start with the assumptions (g1)--(g4) as they can be verified directly from a given form of $\g$ and this verification is easier than showing that the corresponding $\balfa$ is a maximal monotone $p$-coercive graph.

As said above, we aim to develop a robust (i.e. large data) theory for parabolic problems with implicit relations between $\j$ and $\nabla u$ of the form~\eqref{4} assuming that $\g$ fulfills (g1)--(g4). Since the tools we are using are not restricted to scalar problems, we develop the theory for general systems. Thus, instead of considering a vector valued $\g$, we impose our assumption on its tensorial analogue $\cG$. Hence, we assume that for some $p\in (1,\infty)$ and $N\in \mathbb{N}$, the function~$\cG$ and its derivatives $\cG_{\S}$ and $\cG_{\DD}$, defined in an analogous way as $\g_{\j}$ and $\g_{\dd}$, fulfill
\begin{itemize}
\item[(G1)] $\cG \in \mathcal{C}^{0,1}(\R^{{N\times d}} \times \R^{{N\times d}})^{{N\times d}}$ and $\cG(\0,\0)=\0$;
\item[(G2)] for almost all $(\S, \DD) \in \R^{{N\times d}} \times \R^{{N\times d}}$:
\begin{equation*}
\begin{split}
&\cG_{\S}(\S,\DD)\ge 0, \quad \cG_{\DD}(\S,\DD)\le 0, \quad \cG_{\S}(\S,\DD)-\cG_{\DD}(\S,\DD)>0, \\
&\textrm{and } \quad \cG_{\DD}(\S,\DD)(\cG_{\S}(\S,\DD))^T\le 0;
\end{split}
\end{equation*}
\item[(G3)] one of the following holds:
\begin{equation*}
\begin{aligned}
&\text{either} &&\forall \DD\in \mathbb{R}^{{N\times d}} \quad \liminf_{|\S|\to +\infty} \cG(\S, \DD) : \S > 0\\
&\text{or} &&\forall \S\in \mathbb{R}^{{N\times d}}\quad \limsup_{|\DD|\to +\infty} \cG(\S, \DD) : \DD< 0;
\end{aligned}
\end{equation*}
\item[(G4)] there exist $c_1, c_2>0$ such that for all $(\S,\DD)\in \R^{{N\times d}} \times \R^{{N\times d}}$ fulfilling~$\cG(\S,\DD)=\0$ we have
\begin{equation*}
\S :\DD \ge c_1 (|\S|^{p'} + |\DD|^p) -c_2.
\end{equation*}
\end{itemize}
Recall that the constitutive equation $\cG(\S,\DD)=\0$ can be replaced by $-\cG(\S,\DD)=\0$. Then all inequalities in (G2) and (G3) have the opposite signs except the last inequality in (G2). This ambiguity could be fixed for example by requiring that $\cG$ is such that the first condition in (G2) holds.

For simplicity, in what follows, we restrict ourselves to homogeneous boundary data, and then the vector-valued analogue of \eqref{1P} reads as follows:
\begin{equation}\label{problem}
\begin{aligned}
\text{For any given } &\Omega\subset \R^d,~T>0,~ {p\in(1, \infty),~} \u_0:\Omega \to \R^N,~ \f: Q\to \R^N  \text{ and } \\
&\cG:\R^{{N\times d}} \times \R^{{N\times d}} \to  \R^{{N\times d}} \textrm{ satisfying}
 \quad\text{(G1)--(G4)}, \\ \text{find a couple }   &\u: Q \to \R^N \text{ and } \S: Q \to \R^{{N\times d}}\text{ solving the problem }\\
&\begin{aligned}
\pt \u - \diver \S &= \f &&\text{in } Q, \\
\cG(\S, \nabla \u) &= \0 &&\text{in } Q, \\
\u &= \0 &&\text{on } \Sigma_D, \\
\S \n &= \0 &&\text{on } \Sigma_N, \\
 \u(0,\cdot) &= \u_0 &&\text{in } \Omega.
\end{aligned}
\end{aligned}
\end{equation}
Considering a system of equations is of real importance. Indeed, all models depicted in Table~\ref{tab:1} are of the so-called diagonal form. However, in many real applications, one has to deal with systems of equations that contain non-diagonal terms in order to capture observed physical effects. Maxwell--Stefan systems may serve as prototypical examples. The Maxwell--Stefan system describes the diffusive transport of multicomponent mixtures (see for example \cite{giovangigli1999}, \cite{Bothe2011}, \cite{jungel}; regarding the notation we follow \cite{jungel}), where the governing equations for the concentrations $u_{\nu}:(0,T)\times \Omega \to \R$, $0\le u_{\nu} \le 1$, take, for $\nu=1,\dots, N$, $N\ge 2$, the form
\begin{align}
\pt u_{\nu} - \diver \j_{\nu} &= r_{\nu}(\u),  \label{m-s1}\\ 
\nabla u_{\nu} &= \sum_{\mu=1, \mu\neq \nu}^{N} \alpha_{\mu \nu} \left( { u}_{\mu} \j_{\nu} - { u}_{\nu} \j_{\mu} \right). \label{m-s1*} 
\end{align}
Here, the constants $\alpha_{\nu \mu}$ are all positive for $\nu\neq \mu$ and $\u:=(u_1, \dots, u_{N})$. Denoting $\dd_{\nu}:=\nabla u_{\nu}$, $\DD:=(\dd_1, \dots, \dd_{N})^T$ and
$\S:=(\j_1, \dots, \j_{N})^T$, the equations \eqref{m-s1*} can be written in the form
\begin{align}\label{m-s2}
\DD &= \mathbb{B}(\u) \S,
\end{align}
where $\mathbb{B}$ is an $N\times N$-matrix. It is shown in Appendix~\ref{App2} that $\G:\R^{N\times d}\times\R^{N\times d}$ defined through
$$
\G(\S,\DD) = \mathbb{B}(\u) \S - \DD
$$
{fulfills} the conditions (G1)--(G3). Of course, since the matrix $\mathbb{B}$ and the right-hand side of the first set of equations in \eqref{m-s1} depends on the unknown $\u$, the theory developed below cannot be applied to \eqref{m-s1} and the existence analysis of the relevant initial-boundary-value problems must be done in a more delicate way, for which we refer to the above studies \cite{Bothe2011} and \cite{jungel}.  A thermodynamical basis for diffusive transport of multicomponent mixtures that go beyond Maxwell-Stefan systems \eqref{m-s2} and are more appropriate for the realistic description of mixtures and that belong to the class of fully implicit relations is developed in a recent study~\cite{BoDr20}.

We conclude this introductory section by formulating freely the main result of this study:
\begin{align*}
  &\textrm{ For any } \Omega\subset \R^d,~T>0,~{p \in (1,\infty),}~ \u_0,~\f \textrm{ and } \cG \textrm{ satisfying (G1)--(G4),} \\
	&\textrm{ there is a pair } (\u, \S) \textrm{ solving } \eqref{problem}.
\end{align*}

The structure of the remaining parts of the paper is the following. In Section \ref{Sec-2} we provide the precise formulation of our main result and summarize its novelties. We also formulate an analogous result for the boundary-value problem in the elliptic (i.e. time-independent) case. Then, in Section~\ref{maxmon}, we recall the concept of maximal monotone $p$-coercive graph and establish its connection to the implicit constitutive equation $\cG(\S,\DD)=\0$. In particular, we show that the assumptions (G1)--(G4) imply that the null points of $\cG$ form a maximal monotone $p$-coercive graph. In Section \ref{graph-eps}, we construct the appropriate approximating $2$-coercive graphs, parametrized by $\varepsilon$, that are shown to be Lipschitz continuous and uniformly monotone mappings. { In fact, we offer three possible constructions of such approximations. These constructions are} made very explicitly by using an algebraic structure of monotone graphs.  We then investigate the convergence properties between the maximal monotone $p$-coercive graphs and { two of the suggested} approximations and we also add a few additional results, which are useful on their own, to this section. Then, in Section~\ref{proofs}, we prove the main theorem using the approximation of the null points of $\cG$ by the Lipschitz continuous and uniformly monotone single-valued mappings constructed and analyzed in the previous section and letting the approximation parameter $\varepsilon$ tend to zero. The relevant standard theory regarding the well-posedness of the approximate problems, based on the classical Minty method~\cite{Minty}, is, for the sake of completeness, proved in Appendix~\ref{App3}. As indicated above, in Appendices~\ref{App1} and~\ref{App2}, we focus on several constitutive equations that belong to the class $\cG(\S,\DD)=\0$ and verify that they fulfill the structural assumptions (G1)--(G4).

\section{Main result} \label{Sec-2}

Before we state the main result of the paper, we fix some notation.  We recall that throughout the whole paper $\o \subset \R^d$ is a Lipschitz domain (with two mutually non-intersecting essential parts $\Gamma_D$ and $\Gamma_N$ of the boundary $\partial \o$), as defined in Section \ref{intro} after \eqref{1P}. For $t \in (0,T]$, we denote $Q_t:=[0,t) \times \o$ and we also set $Q := Q_T$.
The abbreviation \emph{a.a.}~$t$ stands for \emph{almost all}~$t$. 

We employ small boldfaced letters to denote vectors and bold capitals for tensors. We do not relabel the original sequence when selecting a subsequence. The symbols $\j\cdot\dd$ and $\S:\DD$ stand for the scalar product of vectors $\j$ and $\dd$ or tensors $\S$ and $\DD$, respectively. In a time-space domain, the standard differential operators, such as gradient ($\nabla$) and divergence ($\diver$), are always related to the spatial variables only. Also, we use standard notation for partial ($\partial_{\cdot}$) and total ($\frac{d}{d \cdot}$) derivatives. Generic constants, that depend only on data, are denoted by~$C$ and may vary from line to line.

For a Banach space~$X$, its dual is denoted by $X^*$. For $x \in X$ and $x^*\in X^*$, the duality is denoted by $\langle x^*, x \rangle_X$. For $p \in [1,\infty]$, we denote $(L^p(\o), \|\!\cdot\!\|_{L^p(\Omega)})$ and $(W^{1,p}(\o),\|\!\cdot\!\|_{W^{1,p}(\Omega)})$ the corresponding Lebesgue and Sobolev spaces with the norms defined in standard way,
\begin{align*}
\| f \|_{L^p(\o)} &:=
\begin{cases}
\left( \io |f|^p \d x \right)^{\frac{1}{p}} &\text{if } p \in [1, \infty), \\
\esssup_{x \in \o} |f(x)| &\text{if } p = \infty,
\end{cases} \\
\hspace{2cm}\| f \|_{W^{1,p}(\o)} &:= \|f\|_{L^p(\o)} +  \|\nabla f \|_{L^p(\o)}.
\end{align*}
Bochner spaces are denoted by $L^p(0,T;X)$ and we set
\begin{equation*}
\C([0,T];X) := \{f \in L^{\infty}(0,T;X); [0,T] \ni t^n \to t \implies f(t^n) \to f(t)~\text{ strongly in } X\}.
\end{equation*}
We use the notation $L^p(\o;\R^N)$ and $L^p(\o;\R^{{N\times d}})$ for Lebesgue spaces of vector- or matrix-valued functions, respectively.

Next, we define the function spaces related to our setting. We set
\begin{equation}
\begin{aligned}
V_p &:= \{\u; \u\in W^{1,p}(\o; \R^N) \cap L^2(\o;\R^N), \; \u=\0 \textrm{ on } \Gamma_D\},\\
H&:=L^2(\o; \R^N),\\
V_p^* &:= (V_p)^*,
\end{aligned}\label{spaces}
\end{equation}
and equip the space $V_p$ {with} the norm\footnote{{ Note that in case $p\ge 2d/(d+2)$, we have, due to the Sobolev embedding and the Poincar\'{e} inequality, that $V_p=W^{1,p}_{\Gamma_D}(\Omega;\R^N)$, where $ W^{1,p}_{\Gamma_D}(\Omega;\R^N):=\{\u; \u\in W^{1,p}(\o; \R^N), \; \u=\0 \textrm{ on } \Gamma_D\}.$}} $\|\u\|_{V_p}:= \|\nabla \u\|_{L^p(\o{;\R^{N\times d}})} + \|\u\|_{L^2(\o{;\R^{N}})}$. Then, for any $p\in (1,\infty)${,}
\begin{equation}\label{gelfand}
V_p \hookrightarrow H \equiv H^*  \hookrightarrow V_p^*
\end{equation}
and both embeddings are continuous and dense. Therefore, these spaces form a Gelfand triple. For simplicity, we also set $V:=V_2$ and $V^*:=V_2^*$. Note that $V$ and $H$ are Hilbert spaces.  Also, the duality in $V_p$ is defined via
\begin{equation}\label{dualityVf}
\langle \f, \vp \rangle_{V_p}:= \lim_{k\to +\infty} \io \f^k \cdot \vp \d x
\end{equation}
for any $\vp \in V_p$, where $\{\f^k\}_{k\in \mathbb{N}}$ is a sequence in $H$ converging to $\f$ in $V^*_p$. Note that in the case when $\f\in L^2(\o;\R^N)$, this definition just means
\begin{equation}\label{fvL2}
\langle \f, \vp \rangle_{V_p}=\io \f\cdot \vp \d x.
\end{equation}

Having introduced the notation, we can now formulate the main result of the paper.
\begin{theorem}\label{result}
Let $\Omega \subset \R^d$ be a Lipschitz domain, $T>0$ and $p\in(1,\infty)$. Let $\f \in L^{p'}(0,T; V_p^*)$ and $\u_0 \in H$. Assume that $\G$ satisfies (G1)--(G4). Then, there exists a weak solution to the problem~\eqref{problem}, i.e. there exists a couple $(\u, \S)$ fulfilling
\begin{align*}
\u &\in  L^p(0, T; V_p) \cap  \C([0, T]; H), \\
\pt \u &\in L^{p'}(0, T; V_p^*),\\
\S &\in L^{p'}(Q; \R^{{N\times d}}),
\end{align*}
so that
\begin{subequations}
\begin{align}\label{WFSr}
\langle \pt \u,  \vp \rangle_{V_p} + \io \S : \nabla \vp \d x  &= \langle \f,\vp \rangle_{V_p} \quad \textrm{ for a.a. } t\in (0,T) \textrm{ and for all } \vp\in V_p, \\
\G(\S,\nabla \u)&=\0 ~\textrm{ almost everywhere in }Q,\label{pepa_3ce}
\end{align}
and the initial condition is attained in the strong sense, i.e.,
\begin{equation}\label{ICSr}
\lim_{t\to 0_+} \|\u(t) - \u_0\|_H = 0.
\end{equation}
In addition, $\u$ is uniquely determined.
\end{subequations}
\end{theorem}

Several comments regarding this result and its novelties are in order:

\bigskip

{\bf \emph{(i)}} \ As explicit constitutive equations (as those listed in Table~\ref{tab:1}) represent important subparts of implicit constitutive equations, there are plenty of (even classical) examples in various areas of science (solid and fluid mechanics, heat transfer, chemistry, electro-magnetism, etc.), including Hooke's, Fourier's, Fick's laws and their various non-linear generalizations, that are covered by the equation $\cG(\S,\DD) = \0$. The systematic study of implicit constitutive equations is however more recent and goes back to works by Rajagopal (see \cite{krr1,krr2} for the original papers in elasticity and fluid mechanics\footnote{Note that when reducing the governing equations for incompressible implicitly constituted fluids to simple shear flows one obtains a scalar version of the problem \eqref{problem} studied here.}, and also the survey papers \cite{BMRS2014,KRRS2016} and Section 4.5 in \cite{MP2018}). This new viewpoint {at} constitutive theory in terms of implicit equations has stimulated the analysis of general problems in fluid mechanics in \cite{BGMS1,BGMS2,BGMS3}, where both the stationary and evolutionary situations have been treated{, incorporating even the possibility that the constitutive equation varies with~$t$ and~$x$, i.e. instead of \eqref{pepa_3ce}, considering $\cG(t,x,\S,\nabla \u)=\0$ a.e. in $Q$}. However, there are, in our opinion, two shortcomings in the theory developed in \cite{BGMS1,BGMS2,BGMS3} for incompressible fluid flow problems and in \cite{BMZ} for flows in porous media, and used in some subsequent studies.

{ First,} the proof is highly nonconstructive as it applies standard mollification (to the selection) by convolution, which is very hard to implement numerically, see the recent studies \cite{DieningKreuzerSuli_2013,KreuzerSuli_2016,SuliTscherpel_2020,FarrellOrozcoSuli_2020} devoted to the analysis of finite element discretizations of implicitly constituted fluid flow problems. In our proof below, we do not use convolution at all. Instead{,} we introduce { two very simple algebraic modifications/approximations of $\G$, which are easy to implement.}
In addition, these approximation schemes always lead to a setting in Hilbert spaces on which the approximation graph is Lipschitz continuous and uniformly monotone, which { is the simplest nonlinear setting for PDE and numerical analysis and computation}. { Both approximations differ from the Yosida approximation and the approximation used in Francfort et al. \cite{FrMuTa04}. We provide more details in Section \ref{graph-eps}.}

{ Second, the general theory developed in \cite{BGMS1,BGMS2} for constitutive equations of the type $\cG(t,x,\S,\DD)=\0$ assumes the existence of a \emph{Borel measurable} selection, and the whole proof stems from the assumed a~priori existence of such a selection. More precisely, in the works \cite{BGMS1,BGMS2}, the existence of a selection that is Lebesgue measurable with respect to $(t,x)$ and Borel measurable with respect to one of the tensorial quantities $\S$ or $\nabla \u$ is a~priori granted by the assumption (A5), see \cite[page 110]{BGMS1} or \cite[page 2765]{BGMS2} and this is then substantially used for the construction of the solution. In this study, we do not require the existence of a measurable selection and we do not even use the selection operator at all, but we build the whole theory on the explicit forms for $\varepsilon$-approximations of a general implicit constitutive equation. Moreover, and this seems to be the most essential improvement, the way how we build a solution for a $(t,x)$-independent graph can be easily extended also to a $(t,x)$-dependent graph generated by $\cG(t,x,\S,\DD)=\0$. We conjecture that if the function $\cG$ is measurable with respect to time and space for all $\S$, $\DD$ and continuous with respect to the third and fourth variable a.e. in $Q$ (i.e. $\cG$ is  a Carath\'{e}odory function), then the existence of a weak solution to the initial-boundary-value problem can be established with no requirement on the existence of a measurable selection. Note that the assumption that $\cG$ is a Carath\'{e}odory function is usually easy to check, while the assumption (A5) in \cite{BGMS1,BGMS2} or the assumption concerning the existence of the Lipschitz mapping $\varphi$ in \cite[Assumptions (2.2)--(2.5)]{FrMuTa04} may be more difficult to verify.}

\smallskip
{\bf \emph{(ii)}} \ Although our proof uses the concepts of monotone and maximal monotone mappings/graphs, we formulate the result without using these terms. This is due to the fact that we have found easy-to-verify conditions on the function $\cG$, see the conditions (G1)--(G4)  (almost) characterizing that the corresponding $\mathcal{A}$ is a maximal monotone $p$-coercive graph{, see the last part of Section~\ref{maxmon} for details.}

\smallskip
{\bf\emph{(iii)}} \ If we identify the null points of $\G$ with a set $\mathcal{A}$, a subset of the Cartesian product $\R^{{N\times d}} \times \R^{{N\times d}}$, then one can reformulate the problem in terms of~$\mathcal{A}$, which leads to the theory of monotone mappings. This theory goes back to the seminal work~\cite{Minty}. This theory when  further extended to the analysis of partial differential equations or to problems of the calculus of variations has however relied on sticking to the assumption that the flux is a (possibly multivalued) function of the gradient of the unknown function. {This} concerns mostly the first row in Table~\ref{tab:1}. Several concepts such as multivalued sets, subdifferential calculus, variational inequalities, differential inclusions, etc. have been used to set-up a rigorous mathematical background for relevant problems. One of the aims of this study is to avoid using such concepts and provide, what is in our opinion, a more simple mathematical description; see also the point (iv) below. 

\smallskip
{\bf \emph{(iv)}} \ There are many results where the null points of $\G$ are assumed to be described by a convex potential. To be more precise, if one assumes that there exist a convex $\Phi:\mathbb{R}^{{N\times d}}\to \R$ and its convex conjugate $\Phi^*:\R^{{N\times d}} \to \R$ such that
\begin{equation} \label{potential}
\G(\S,\DD)=\0 \qquad\Longleftrightarrow \qquad \S : \DD = \Phi(\S)+\Phi^*(\DD)
\end{equation}
then the condition ``$\G(\S,\nabla \u)=\0$ almost everywhere in $Q$" stated in Theorem~\ref{result}, can be equivalently\footnote{{Indeed, if we set $\vp:=\u$ in \eqref{WFSr} and compare it with \eqref{foridiots}, we see that $\io \Phi(\S)+\Phi^*(\nabla \u) \d x \le \io \S:\nabla \u \d x$. Then, due to Young's inequality, one concludes that $\Phi(\S)+\Phi^*(\nabla \u) =\S:\nabla \u$ a.e. in $\Omega$, and consequently the assumption \eqref{potential} gives $\G(\S,\nabla \u)=\0$ a.e. in $\Omega$.}} replaced by the following inequality
\begin{align}\label{foridiots}
\frac12 \dt \|\u\|_{H}^2+ \io \Phi(\S)+\Phi^*(\nabla \u) \d x  \le \langle \f,\nabla \u \rangle_{V_p}.
\end{align}
It is then obvious that due to the convexity of~$\Phi$ and~$\Phi^*$, one can usually pass to the inequality~\eqref{foridiots} without any major difficulties. Unfortunately, such a procedure works only in the case~of~\eqref{potential}, which decreases the applicability of such a theory significantly. The second (and more important) limitation of this approach is that, for consistency, it requires the possibility of using~$\u$ as a test function in~\eqref{WFSr}, which is typically not the case in problems arising in fluid dynamics, see e.g. \cite{AF}, where an inequality of the type~\eqref{foridiots} is used to define the notion of solution.

\smallskip
{\bf \emph{(v)}} We wish to emphasize that there are interesting constitutive equations of the form $\cG(\S,\DD)= \0$ that generate a \emph{non-monotone} graph and consequently the analysis of the corresponding problems is not covered by the theory developed in this paper. In fact, a sounding analysis for problems with non-monotone graphs is a challenging open problem. We refer to~\cite{LerouxKRR,JMPT_2019} for more details.

\smallskip {
{\bf \emph{(vi)}} The subclasses of constitutive equations listed in Table~\ref{tab:1} were taken from Blechta et al.~\cite{blechta2020} where a classification of incompressible fluids has been made and where in addition a PDE analysis (a long-time and large-data existence theory) was developed for one novel class of fluids that emerge from this classification, namely for activated Euler fluids. (These are fluids that behave as Euler (inviscid) fluids until the activation takes place and then the fluid responds as a Navier-Stokes fluid or, more generally, a fluid of a power-law type.) The study~\cite{blechta2020} differs significantly from the results presented in this paper. First of all, to construct an approximation of one specific constitutive model is always easier than to build the approximations for a general class of implicit equations. In fact, the construction of approximation for activated Euler fluids is simple: one adds a viscous stress term for a Newtonian fluid with small viscosity $\varepsilon>0$ to the model. In addition, both activated Euler fluids as well as their approximations are examples of explicit models where the Cauchy stress depends the velocity gradient explicitly. This makes the analysis, in comparison with fully implicit constitutive equations, simpler. On the other hand, Blechta et al.~\cite{blechta2020} studied models where besides a nonlinear constitutive relation there is an additional nonlinearity due to the convective term. In \cite{blechta2020}, a mathematical theory for steady and unsteady flows of activated Euler fluids, subject to different types of boundary conditions, is developed for a large set of parameters $p$, namely $p>6/5$, characterizing the behavior of fluids for large values of the velocity gradient. Special and systematic attention is also devoted to the boundary conditions. In this study,
focusing on a potentially broader class of applications, we go beyond the realm of the fluid mechanics problems (by omitting the incompressibility constraint and the convective term), but we pay particular attention on the development of the analytical techniques for a general class of implicit equations.}
\bigskip

We complete this section by stating the result for the time-independent case. We however do not give the complete proof of this result here since it is easier than in the time-dependent situation and in fact the proof can be deduced from the detailed proof of Theorem \ref{result} directly by eliminating the steps that are {there} due to the dependence of the quantities on time.
\begin{theorem}
\label{result-2}
Let $\Omega \subset \R^d$ be a Lipschitz domain, $\Gamma_D\neq \emptyset$, $p\in(1,\infty)$ and  $\f \in (W^{1,p}_{\Gamma_D}(\o; \R^N))^*$. Assume that $\G$ satisfies (G1)--(G4). Then, there exists a couple $(\u, \S)$ fulfilling
\begin{align*}
\u &\in W^{1,p}_{\Gamma_D}(\Omega;\R^N), \\
\S &\in L^{p'}(\Omega; \R^{{N\times d}}),\\
\G(\S,\nabla \u)&=\0 \textrm{ almost everywhere in }\Omega,
\end{align*}
which satisfies for all $\vp\in W^{1,p}_{\Gamma_D}(\Omega;\R^N)$
\begin{subequations}
\begin{equation}\label{WFSr-s}
 \io \S : \nabla \vp \d x  = \langle \f,\vp \rangle_{W^{1,p}_{\Gamma_D}(\Omega;\R^N)}.
\end{equation}
\end{subequations}
In addition, if $\G$ satisfies
\begin{itemize}
\item[(G2)$^*$] for any $(\S_1, \DD_1), (\S_2, \DD_2) \in \R^{{N\times d}} \times \R^{{N\times d}}$ fulfilling~$\cG(\S_i,\DD_i)=\0$ and  $\DD_1\neq \DD_2$:
\begin{equation*}
(\S_1 - \S_2):(\DD_1-\DD_2)> 0,
\end{equation*}
\end{itemize}
then the solution $\u$ is unique.
\end{theorem}
The above result does not include the purely Neumann problem, i.e. $\Gamma_N=\partial \Omega$. Nevertheless, the existence statement of Theorem~\ref{result-2} remains true provided that the right-hand side fulfills the necessary compatibility condition
\begin{equation*}
 \langle \f,\vp \rangle_{W^{1,p}(\Omega;\R^N)}=0 \qquad  \textrm{ for all constant } \vp\in \R^N.
\end{equation*}
Moreover, the uniqueness result holds true if restricted to functions with prescribed mean value.

\section{Null points of \texorpdfstring{$\cG$}{G} and maximal monotone graphs}\label{maxmon}

In this part, we identify the null set of $\G$ with a subset $\A$ of $\R^{{N\times d}}\times \R^{{N\times d}}$ and show that the assumptions (G1)--(G4) delimiting the structure of $\cG$ imply that $\A$ is a maximal monotone $p$-coercive graph (see the definition below). { Moreover, we show in which sense  the assumptions (G1)--(G4) are necessary in order to guarantee the maximal monotonicity of $\A$.} Before doing so, we develop a generalized monotone operator theory following the original work~\cite{Minty} as well as \cite{AlbAmb} and \cite{BrCrPa} where a similar but less general approach is used.

Let us start by recalling the notion of maximal monotone graph.
\begin{definition}[Maximal monotone $p$-coercive graph]\label{maxmongrafA}
Let $p \in (1, \infty)$ and $p' := \frac{p}{p-1}$. We say that a subset $\A$ of  $\R^{{N\times d}} \times \R^{{N\times d}}$ is a maximal monotone $p$-coercive graph if
\begin{itemize}
\item[(A1)] $(\0,\0) \in \A$;
\item[(A2)] For any $(\S_1,\DD_1), (\S_2,\DD_2) \in \A$
\begin{equation*}
(\S_1 - \S_2):(\DD_1 - \DD_2) \ge 0;
\end{equation*}
\item[(A3)] If for some $(\S,\DD) \in \R^{{N\times d}} \times \R^{{N\times d}}$ and for all $(\oS,\oD) \in \A$
\begin{equation*}
(\S - \oS):(\DD - \oD) \ge 0,
\end{equation*}
then $(\S,\DD) \in \A$;
\item[(A4)] There exist $C_1, C_2 >0$ such that for all $(\S,\DD) \in \A$
\begin{equation*}
\S :\DD \ge C_1(|\S|^{p'} +|\DD|^{p})-C_2.
\end{equation*}
\end{itemize}
\end{definition}
The condition (A1) means that $\A$ passes through the origin, (A2) states that the graph $\A$ is \underline{monotone}, while (A3) states that $\A$ is \underline{maximal monotone}, i.e. $\A$ \emph{cannot be extended to a properly larger domain while preserving its monotoneity}\footnote{The text in italics is verbatim citation from \cite{Minty}.}. Finally, (A4) states that the graph $\A$ is \underline{$p$-coercive}.
\begin{remark} For further generality, one could replace (A4) with the following condition:
\begin{itemize}
\item[(A4$^*$)] There exist $c^*, c_*>0$ and a Young function $\psi$ such that for all $(\S,\DD) \in \A$
$$
\S :\DD \geq c^*(\psi(|\DD|) + \psi^*(|\S|))-c_*.
$$
\end{itemize}
Here, $\psi:\R \to \R^+$ is a Young function, i.e. $\psi$ is an even continuous convex function such that
\begin{equation*}
\lim_{s \to 0_+} \frac{\psi(s)}{s}=0 ~\text{ and }~ \lim_{s \to +\infty} \frac{\psi(s)}{s}=+\infty,
\end{equation*}
and the convex conjugate function $\psi^*$ is defined as the Legendre transform of $\psi$, i.e.,
\begin{equation*}
\psi^*(s) := \sup_{l\in\R}(s\cdot l - \psi(l)).
\end{equation*}
The study of the models related via \emph{maximal monotone $\psi$-graphs}, i.e. the graphs satisfying (A1)--(A3) and (A4$^*$), are of interest; however, such an extension is nowadays rather routine and it is not included in this work.
\end{remark}



Next, we prove an auxiliary result adopted from Minty~\cite{Minty} and~\cite[Proposition 1.1 (applied to dimension ${N\times d}$)]{AlbAmb} and adjusted to our situation. More precisely, having a monotone graph $\A$ we can identify it with two possibly multivalued (monotone) mappings $\S^*$ and $\DD^*$, each defined on a subset of $\R^{{N\times d}}$ through
\begin{equation}\label{multivalueD}
(\oS, \oD) \in \A ~\Longleftrightarrow~ \oS \in \S^*(\oD)~\Longleftrightarrow~ \oD \in \DD^*(\oS).
\end{equation}
Then, $\S^*$ is maximal monotone if and only if $\S^* + \e \I$ is onto for any $\e \in (0,1]$ and $\DD^*$ is maximal monotone if and only if $\DD^* + \e \I$ is onto for any $\e \in (0,1]$. In the next lemma, we will prove the first statement noting that the proof of the second equivalence can be done in the same way just by interchanging the role of $\DD$ and $\S$.

\begin{lemma}\label{onto}
Let $\e \in (0,1]$ and let $\A$ be a monotone graph identified with $\S^*$ via~\eqref{multivalueD}, i.e. $\A$ satisfies (A1) and (A2).
Then, $\A$ is maximal monotone, i.e. $\A$ satisfies (A3), if and only if the mapping $\S^* + \e \I$ is onto.
\end{lemma}
\begin{proof} We split the proof into two steps. In the first step, we show that if $\A$ is a maximal monotone graph then $\S^* + \I$ is onto (i.e. the domain of $(\S^* + \I)^{-1}$ is $\R^{{N\times d}}$). In the second step we prove the converse implication. In the proof, we restrict ourselves to the case $\e=1$ as the proof can be easily extended to the case $\e\in(0,1)$.

\paragraph{\bf Step 1.} We start by defining a set
\begin{equation}\label{Im}
\Im:=\{ \z \in \R^{{N\times d}}; \exists (\S,\DD)\in \A, \z = \S+\DD\}
\end{equation}
and our goal is to define $(\S^* + \I)^{-1}$ on $\Im$ and show that $\Im=\R^{{N\times d}}$.

\smallskip

\paragraph{\underline{\emph{$\Im$ is nonempty and closed}}} As $\0 \in \Im$, the set $\Im$ is nonempty. In addition, $\Im$ is closed, i.e. the following condition holds:
\begin{equation}\label{closed}
\z_j \in \Im \text{ for } j \in \N, ~\z_j \to \z \text{ in } \R^{{N\times d}} \text{ as } j\to\infty ~\implies~ \z \in \Im.
\end{equation}
Indeed, since $\z_j \in \Im$ there exist $(\S_j, \DD_j)\in \A$ such that $\z_j = \S_j + \DD_j$ for every $j$ and since
$\{\z_j\}$ is bounded, it follows from (A2) and (A1) that
\begin{equation*}
|\S_j|^2 + |\DD_j|^2 \leq |\S_j+\DD_j|^2= |\z_j|^2 \leq C.
\end{equation*}
Thus, the sequences $\{\S_j\}$ and $\{\DD_j\}$ are bounded and there exists a couple  $(\S,\DD)$ such that for a subsequence (that we do not relabel) $\S_j \to \S$ and $\DD_j \to \DD$ as $j \to \infty$. As $\z_j\to \z$, we conclude that $\z= \S + \DD$. Due to (A2) and the fact that $(\S_j,\DD_j)\in \A$ we also get
\begin{equation*}
(\S - \vec{A}):(\DD - \vec{B}) = \lim_{j\to \infty} (\S_j - \vec{A}):(\DD_j - \vec{B}) \geq 0\qquad \textrm{ for all } (\vec{A}, \vec{B})\in \A.
\end{equation*}
Then, by virtue of the maximality (A3), $(\S, \DD) \in \A$. Hence, $\z \in \Im$.

\smallskip

\paragraph{\underline{\emph{Definition of the mapping $(\S^* + \I)^{-1}$}}} On $\Im$, we define
\begin{equation}\label{DIinv}
(\S^* + \I)^{-1} (\z) := \{ \DD \in \R^{{N\times d}};\;  \exists \S, \; (\S, \DD) \in \A,\;  \S+\DD = \z\}
\end{equation}
and show in the following lines that $(\S^* + \I)^{-1}$ is a well-defined single-valued mapping and $\Im = \R^{{N\times d}}$.

We first check that $(\S^* + \I)^{-1}$ is $1$-Lipschitz on $\Im$. To this end, let us take $\z_1, \z_2 \in \Im$, $\z_1 \neq \z_2$, and $\DD_1 \in (\S^* + \I)^{-1} (\z_1)$, $\DD_2 \in (\S^* + \I)^{-1} (\z_2)$. Then, with the help of (A2), we have
\begin{equation*}
(\z_1 - \z_2):(\DD_1 - \DD_2)= (\S_1 - \S_2 +\DD_1 - \DD_2):(\DD_1 - \DD_2) \geq |\DD_1 - \DD_2|^2.
\end{equation*}
Then, $1$-Lipschitz continuity directly follows from the Cauchy-Schwarz inequality.

\smallskip

\paragraph{\underline{$\Im = \R^{{N\times d}}$}} Next, for contradiction, assume that $\Im \subsetneq \R^{{N\times d}}$. We then define an auxiliary function
\begin{equation}\label{Fz}
\F(\z):= \z - \sqrt{2}(\S^* + \I)^{-1}(\sqrt{2}\z), \qquad \z\in \frac{1}{\sqrt{2}}\Im
\end{equation}
and observe that $\F$ is also $1$-Lipschitz on $\frac{1}{\sqrt{2}}\Im$. Indeed, let $\z_1, \z_2 \in \frac{1}{\sqrt{2}}\Im$,
$\z_1 \neq \z_2$ and consider $\S_i, \DD_i$, $i=1,2$, such that $\S_1 + \DD_1 = \sqrt{2}\z_1$, $\S_2 + \DD_2 = \sqrt{2}\z_2$. Then with the help of \eqref{DIinv} and (A2) we observe that
\begin{align*}
|\F(\z_1) - \F(\z_2)|^2 &= |\z_1 - \z_2|^2 + 2|\DD_1 - \DD_2|^2 - 2 \sqrt{2} (\z_1 - \z_2):(\DD_1 - \DD_2) \\
&= |\z_1 - \z_2|^2 - 2(\S_1 - \S_2):(\DD_1 - \DD_2) \leq |\z_1 - \z_2|^2.
\end{align*}
As we assume that $\frac{1}{\sqrt{2}}\Im$ is a proper subset of $\R^{{N\times d}}$, i.e. $\frac{1}{\sqrt{2}}\Im \subsetneq \R^{{N\times d}}$, we can extend $\F$ defined on~$\frac{1}{\sqrt{2}}\Im$ to $\tF$ defined on~$\R^{{N\times d}}$ in such a way that
\begin{equation*}
\tF (\z)=
\begin{cases}
\begin{aligned}
&\F(\z) &&\z\in \frac{1}{\sqrt{2}}\Im,\\
&\text{is } 1\text{-Lipschitz} &&\text{on }\R^{{N\times d}}.
\end{aligned}
\end{cases}
\end{equation*}
Let us now define, for an arbitrary $\sqrt{2}\tz \in \R^{{N\times d}} \setminus \Im$,
\begin{equation}\label{tildeJD}
\tilde{\S}:= \frac{1}{\sqrt{2}}(\tz + \tF(\tz))\quad \textrm{ and } \quad  ~\tilde{\DD}:= \frac{1}{\sqrt{2}}(\tz - \tF(\tz)).
\end{equation}
If we prove that
\begin{equation}\label{golik}
(\tilde{\S}, \tilde{\DD})\in \A,
\end{equation}
then $\tilde{\S} + \tilde{\DD}=\sqrt{2}\tz$ and, due to the definition of $\Im$, $\sqrt{2}\tz \in \Im$, which is the sought contradiction. Thus, the definition domain of $(\S^* + \I)^{-1}$ is $\R^{{N\times d}}$. This means that $\S^* + \I$ is onto.

It remains to verify \eqref{golik}. For this purpose, we use the maximality of $\A$, i.e. the assumption (A3), and show that for all $(\S, \DD) \in \A$ the following holds:
\begin{equation}\label{golik2}
(\tilde{\S}-\S):(\tilde{\DD}-\DD)\ge 0.
\end{equation}
Taking an arbitrary $(\S, \DD) \in \A$ and setting $\sqrt{2}\z:=\S+\DD$, we observe that $\sqrt{2}\z \in \Im$. Then, by virtue of the definition $\F$ (see \eqref{Fz}), we have
\begin{equation}\label{beztilde}
\S= \frac{1}{\sqrt{2}}(\z + \F(\z))\quad \textrm{ and } \quad \DD= \frac{1}{\sqrt{2}}(\z - \F(\z)).
\end{equation}
Using then \eqref{tildeJD}, \eqref{beztilde} and the fact that $\tF$ is $1$-Lipschitz continuous on $\R^{{N\times d}}$, we obtain
\begin{align*}
2(\tilde{\S} - \S):(\tilde{\DD} - \DD) &= (\tz -\z+ \tF(\tz)-\F(\z)):(\tz -\z-(\tF(\tz)-\F(\z))) \\
&= |\tz -\z|^2 - |\tF(\tz)-\F(\z)|^2 \geq 0,
\end{align*}
which proves \eqref{golik2}. Thus, the proof of \eqref{golik} is complete.

\smallskip

\paragraph{\bf Step 2.} It remains to prove the second implication, i.e. if $\S^*+\I$ is onto, then the monotone graph $\A$ is maximal, i.e. (A3) holds. Let $(\tilde{\S},\tilde{\DD})\in \mathbb{R}^{{N\times d}} \times \R^{{N\times d}}$ be such that for all $(\S,\DD)\in \A$ \eqref{golik2} holds. Our goal is to show that $(\tilde{\S},\tilde{\DD})\in \A$. Since $\S^*+\I$ is onto, we know that for $(\tilde{\S},\tilde{\DD})$ there exists a couple $(\S,\DD)\in \A$ such that
\begin{equation}\label{cislo}
(\S^*+\I)(\DD) = \S+\DD = \tilde{\S}+\tilde{\DD}.
\end{equation}
Now, since $(\S,\DD)\in \A$, we use the decomposition~\eqref{cislo} in \eqref{golik2} and deduce that
$$
\begin{aligned}
0&\le  2(\tilde{\S}-\S): (\tilde{\DD} -\DD)=(\tilde{\S}-(\tilde{\S}+\tilde{\DD}-\DD)): (\tilde{\DD} -\DD)+(\tilde{\S}-\S): (\tilde{\DD} -(\tilde{\S}+\tilde{\DD}-\S))\\
&=-|\tilde{\DD} -\DD|^2 -|\tilde{\S}-\S|^2.
\end{aligned}
$$
Consequently, $\tilde{\DD}=\DD$ and $\tilde{\S}=\S$ and $(\tilde{\S},\tilde{\DD})\in \A$, which finishes the proof.
\end{proof}


\begin{lemma}\label{GvsA}
Let $\G$ satisfy assumptions (G1)--(G4). Let
\begin{equation}\label{GimpliA}
\mathcal{A}:=\{(\S,\DD)\in \R^{{N\times d}} \times \R{^{N\times d};} \; \G(\S,\DD)=\0\}.
\end{equation}
Then, $\mathcal{A}$ is a maximal monotone $p$-coercive graph.
\end{lemma}
Here, we would like to emphasize that the assumptions (G1)--(G4) are associated with the implicit constitutive equation $\cG(\S,\DD)=\0$  and it is not evident a~priori that the null set of $\cG$ is a maximal monotone $p$-coercive graph. However, we will show below that the monotonicity condition (A2) is in fact a consequence of (G2) while the maximality (A3) follows from the continuity of~$\G$ and~(G3).

\begin{proof}[Proof of Lemma~\ref{GvsA}] We start the proof with several simple observations. Recalling the definition of $\A$ given in~\eqref{GimpliA}, we see directly that (G1)$\implies$(A1) and (G4)$\iff$(A4). Thus, it remains to verify the conditions~(A2) and~(A3). We show that they follow from (G2), (G3) and the continuity of $\G$. Without loss of generality, we assume here that the second condition in (G3) is fulfilled. In case that the first condition of (G3) was true, we would have to change the role of $\S$ and $\DD$ in the proof below. We split the proof into several steps.

\paragraph{{\bf Step 1.}} We first show that for every $\varepsilon \in (0,1]$ and every $\Z\in \R^{{N\times d}}$ there exists a $\DD\in \R^{{N\times d}}$ such that
\begin{equation}\label{nece}
\G(\Z-\varepsilon\DD,\DD)=\0.
\end{equation}
Once \eqref{nece} is proved and once we show that the graph $\A$ is monotone (which we shall prove in Step 4 below), then \eqref{nece} implies that $\A$ is maximal by means of Lemma~\ref{onto}. Indeed, we need to check that $\S^* +\e\I$ is onto, i.e. that for every $\Z\in \R^{{N\times d}}$ there exists a couple $(\S,\DD)\in \R^{{N\times d}}\times \R^{{N\times d}}$ such that
\begin{equation}\label{pepa998}
\G(\S,\DD)=\0 \qquad \textrm{and} \qquad \S+\e\DD=\Z.
\end{equation}
However, substituting the second relation into the first one, we observe that it is exactly \eqref{nece}.

{To summarize, once we verify \eqref{nece} and show that $\A$ is monotone, the proof of Lemma~\ref{GvsA} is complete.}

\smallskip

\paragraph{\underline{\emph{Proof of \eqref{nece}}}} For arbitrary $\Z_1$, $\Z_2\in \mathbb{R}^{{N\times d}}$, set $\Z_t:=t\Z_1+(1-t)\Z_2$. Then, by virtue of~(G2), $\G_{\S}(\Z_t,\DD)\ge 0$ for all $\DD\in \mathbb{R}^{{N\times d}}$. Consequently
$$
\begin{aligned}
(\G(\Z_1,\DD)-\G(\Z_2,\DD)):(\Z_1 - \Z_2)&= \int_0^1 \dt \G(\Z_t,\DD):(\Z_1-\Z_2)\d t\\
&= \int_0^1 \G_{\S}(\Z_t,\DD)(\Z_1-\Z_2) :(\Z_1-\Z_2)\d t\ge 0.
\end{aligned}
$$
Taking, in particular, $\Z_1=\Z-\varepsilon \DD$ and $\Z_2= \Z$, where $\Z$ and $\DD\in \mathbb{R}^{{N\times d}}$ are arbitrary, it follows that
\begin{equation}\label{oddhad}
\begin{aligned}
-\varepsilon (\G(\Z-\varepsilon \DD,\DD)-\G(\Z,\DD)):\DD \geq 0.
\end{aligned}
\end{equation}
Using then (G3), we observe that for arbitrary $\varepsilon\in (0,1]$ and any $\Z\in \mathbb{R}^{{N\times d}}$
\begin{equation}\label{iiil}
\limsup_{|\DD|\to \infty}\G(\Z-\varepsilon \DD, \DD) : \DD \overset{\eqref{oddhad}}{\le }\limsup_{|\DD|\to \infty}\G(\Z, \DD) : \DD <0.
\end{equation}
Thus,  there exists an $R>0$ such that for all $\DD\in \R^{{N\times d}}$ fulfilling $|\DD|\ge R$, we have
\begin{equation}\label{as1a}
\G(\Z-\varepsilon\DD,\DD):\DD\le 0.
\end{equation}
Having this piece of information, we prove \eqref{nece} by contradiction. We thus assume that for all $\DD\in \R^{{N\times d}}$ such that $|\DD|\le R$ one has that
\begin{equation}\label{contra1}
\G(\Z-\varepsilon\DD,\DD)\neq \0.
\end{equation}
Then, due to \eqref{contra1} and the continuity of $\cG$, the mapping
$$
\DD\mapsto R \frac{\G(\Z-\varepsilon \DD,\DD)}{|\G(\Z-\varepsilon\DD,\DD)|}
$$
is defined in a closed ball of radius $R$, is continuous and maps a closed ball of radius $R$ into itself (in fact it maps the ball of radius $R$ onto its sphere). Consequently, by Browder's fixed point theorem, there is a $\DD\in \R^{{N\times d}}$, $|\DD|= R$, such that
$$
\DD= R \frac{\G(\Z-\varepsilon\DD,\DD)}{|\G(\Z-\varepsilon\DD,\DD)|}.
$$
Taking the scalar product of both sides of this equality with $\DD$ and using \eqref{as1a}, we see that
$$
R^2= \DD : \DD = R \frac{\G(\Z-\varepsilon\DD,\DD):\DD}{|\G(\Z-\varepsilon\DD,\DD)|}\le 0,
$$
a contradiction. Hence, for an arbitrarily given $\Z\in \R^{{N\times d}}$ there is $\DD$ satisfying \eqref{nece}.

\smallskip

\paragraph{{\bf Step 2.}} Next we show that for any couple $(\Z_1, \DD_1),(\Z_2, \DD_2)\in \mathbb{R}^{{N\times d}}\times  \mathbb{R}^{{N\times d}}$ fulfilling, for $i=1,2$, $\G(\Z_i-\varepsilon \DD_i, \DD_i)=\0$, the following condition holds:
\begin{equation}\label{loc-lip}
|\DD_1-\DD_2|\le C(\Z_1,\Z_2,\DD_1,\DD_2)|\Z_1-\Z_2|.
\end{equation}
This means that $\DD$ can be understood as a locally Lipschitz function of $\Z$.

\smallskip
\paragraph{\underline{\emph{Proof of \eqref{loc-lip}}}} Let us denote $\Z_t:=t\Z_1+(1-t)\Z_2$ and $\DD_t:=t\DD_1+(1-t)\DD_2$. Then it follows from the assumption (G2) that
$$
\begin{aligned}
\0&=\G(\Z_1-\varepsilon \DD_1, \DD_1) - \G(\Z_2-\varepsilon\DD_2, \DD_2)\\
&=\int_0^1 \dt \G(\Z_t-\varepsilon \DD_t, \DD_t)\d t\\
&= \int_0^1 \G_{\S}(\Z_t-\varepsilon \DD_t, \DD_t)(\Z_1-\Z_2-\varepsilon (\DD_1-\DD_2))+ \G_{\DD}(\Z_t-\varepsilon \DD_t, \DD_t) (\DD_1-\DD_2) \d t.
\end{aligned}
$$
Consequently,
\begin{equation}\label{pepa1}
\int_0^1 \varepsilon \G_{\S}(\dots)-\G_{\DD}(\dots)\d t \, (\DD_1-\DD_2)=\int_0^1 \G_{\S}(\dots)\d t \, (\Z_1-\Z_2),
\end{equation}
where $(\dots)$ stands for $(\Z_t-\varepsilon \DD_t, \DD_t)$.
Thanks to (G2), we know that $\G_{\S}\ge 0$, $-\G_{\DD}\ge 0$ and $\G_{\S}-\G_{\DD}>0$. Consequently, for arbitrary $\varepsilon >0$ we also have $\varepsilon \G_{\S}-\G_{\DD}>0$, and also
$$
I:=\int_0^1 \varepsilon \G_{\S}(\dots)-\G_{\DD}(\dots)\d t >0.
$$
This means that $I$ is positive definite, and consequently, $I$ is an invertible matrix. It thus follows from \eqref{pepa1} that
\begin{equation*}
\DD_1-\DD_2 =\left(\int_0^1 \varepsilon \G_{\S}(\dots)-\G_{\DD}(\dots)\d t \right)^{-1}\int_0^1 \G_{\S}(\dots)\d t \, (\Z_1-\Z_2)
\end{equation*}
and \eqref{loc-lip} follows.

\smallskip
\paragraph{{\bf Step 3.}}
In this step, we show that for arbitrary $(\Z_1,\DD_1),(\Z_2,\DD_2)\in \mathbb{R}^{{N\times d}} \times \mathbb{R}^{{N\times d}}$ fulfilling, for $i=1,2$, $\G(\Z_i -\varepsilon \DD_i, \DD_i)= \0$, there holds
\begin{equation}
\label{monbase}
(\DD_1-\DD_2):(\Z_1-\Z_2) \ge 0.
\end{equation}

\smallskip

\paragraph{\underline{\emph{Proof of \eqref{monbase}}}} It follows from Step 2, that for the null points of $\G(\Z-\varepsilon \DD, \DD)$, we can understand $\DD$ as a locally Lipschitz mapping of $\Z$ and we can write $\DD(\Z)$. Since $\DD$ is Lipschitz, its derivative $\DD_{\Z}(\Z)$ exists for almost all $\Z$. By applying this derivation to $\cG(\Z-\varepsilon \DD(\Z), \DD(\Z))=\0$, we obtain
$$
\begin{aligned}
\0&={\frac{\mathrm{d}}{\mathrm{d}\Z}} \G(\Z-\varepsilon \DD(\Z), \DD(\Z))\\
&=\G_{\S}(\Z-\varepsilon \DD(\Z), \DD(\Z))(\I- \varepsilon \DD_{\Z}(\Z)) + \G_{\DD}{(}\Z-\varepsilon \DD(\Z), \DD(\Z))\DD_{\Z}(\Z)
\end{aligned}
$$
It follows that
$$
(\varepsilon \G_{\S}(\Z-\varepsilon \DD(\Z), \DD(\Z))-\G_{\DD}(\Z-\varepsilon \DD(\Z), \DD(\Z)))\DD_{\Z}(\Z)= \G_{\S}(\Z-\varepsilon \DD(\Z), \DD(\Z)).
$$
Since the matrix on the left-hand side is regular thanks to the assumption (G2), we observe (we omit writing the dependence on $\Z$ for simplicity) that
\begin{equation}
\DD_{\Z}= (\varepsilon \G_{\S}-\G_{\DD})^{-1} \G_{\S}.\label{Dzfor}
\end{equation}
Our next goal is to show  that
\begin{equation}
\DD_{\Z}\ge 0.\label{signD}
\end{equation}
To do so, consider an arbitrary nonzero $\X \in \mathbb{R}^{{N\times d}}$. Since $(\varepsilon \G_{\S}-\G_{\DD})$ is invertible, we can also define $\Y:= (\varepsilon \G_{\S}-\G_{\DD})^{-T} \X$ and with the help of \eqref{Dzfor} obtain
$$
\begin{aligned}
{(}\DD_{\Z}\X {)} :\X &= {(}(\varepsilon \G_{\S}-\G_{\DD})^{-1} \G_{\S}\X {)}:\X = (\G_{\S}\X) : ((\varepsilon \G_{\S}-\G_{\DD})^{-T}\X)\\
&= {(}\G_{\S} (\varepsilon \G_{\S}-\G_{\DD})^{T}\Y{)} : \Y=((\varepsilon \G_{\S}-\G_{\DD})^{T}\Y ): ((\G_{\S})^T \Y)\\
&=\varepsilon |(\G_{\S})^T \Y|^2 -{(}\G_{\DD}(\G_{\S})^{T}\Y {)}: \Y\ge 0,
\end{aligned}
$$
where the last inequality follows from the third assumption in (G2). Finally, since
$$
\DD(\Z_1)-\DD(\Z_2)= \int_0^1\dt \DD(\Z_t)\d t = \int_0^1\DD_{\Z}(\Z_t)\d t \, (\Z_1-\Z_2)
$$
with $\Z_t:=t\Z_1+(1-t)\Z_2$, we can use \eqref{signD} to deduce that
$$
(\DD(\Z_1)-\DD(\Z_2)):(\Z_1-\Z_2)= \int_0^1\DD_{\Z}(\Z_t)\d t (\Z_1-\Z_2) :(\Z_1-\Z_2)\ge 0
$$
and \eqref{monbase} follows.

\smallskip

\paragraph{{\bf Step 4.}} Finally, it remains to verify that $\A$ is monotone, i.e. (A2) holds. Let $\S_1$, $\S_2$, $\DD_1$ and $\DD_2\in \mathbb{R}^{{N\times d}}$ fulfilling, for $i=1,2$, $\G(\S_i,\DD_i)=\0$ be arbitrary. The aim is to show that
\begin{equation}\label{aux_1}
(\S_1-\S_2):(\DD_1-\DD_2)\ge 0.
\end{equation}
To prove this, we define $\Z_i:=\S_i + \varepsilon \DD_i$, $i=1,2$. Then it follows from the assumptions on $(\S_i,\DD_i)$ that $\G(\Z_i-\varepsilon \DD_i, \DD_i)=\0.$ Hence the inequality \eqref{monbase} from Step 3 is valid for $(\Z_1,\DD_1)$ and $(\Z_2, \DD_2)$. Thus, we can continue as follows
$$
\begin{aligned}
(\S_1-\S_2):(\DD_1-\DD_2)&= (\Z_1-\Z_2 - \varepsilon (\DD_1-\DD_2)):(\DD_1-\DD_2)\\
&=(\Z_1-\Z_2):(\DD_1-\DD_2) - \varepsilon |\DD_1-\DD_2|^2 \ge - \varepsilon |\DD_1-\DD_2|^2,
\end{aligned}
$$
where the last inequality follows from \eqref{monbase}. Since the left-hand side is independent of $\varepsilon$, letting $\varepsilon \to 0_+$, we deduce \eqref{aux_1}. The proof of Lemma~\ref{GvsA} is complete.
\end{proof}

{ Lemma~\ref{GvsA} states that the conditions (G1)--(G4) are sufficient for proving that $\mathcal{A}$, as defined in \eqref{GimpliA}, fulfills (A1)--(A4). It is natural to ask whether the conditions (G1)--(G4) are also necessary, i.e. whether a Lipschitz continuous function $\G$ that generates by \eqref{GimpliA} a set $\mathcal{A}$ satisfying (A1)--(A4) has to fulfill (G1)--(G4). It is evident that (A1) and (A4) are equivalent to (G1) and (G4). It is also obvious that the properties of $\mathcal{A}$ follow from the behavior of $\G$ only in the neighborhood of the null points. Consequently one can redefine $\G$ outside this neighborhood arbitrarily. In particular, as $\mathcal{A}$ satisfies (A2) and (A3), if necessary, one can always redefine $\G$ so that (G3) holds. Finally, regarding (G2), the situation seems to be the most interesting one. Considering $\G(\S,\DD):=|\S - \DD|^2 (\S - \DD)$, which generates a maximal monotone (in fact linear) graph, one observes that the presence of a strict inequality sign (for the null points) in the first three inequalities of (G2) is violated. On the other hand, the last inequality in (G2), which seems to be the most essential, reveals to be a necessary consequence of (A2); this is formulated in the next lemma.
\begin{lemma}\label{GvsAAA}
Let $\G$ be a $\mathcal{C}^{0,1}$ mapping and let $\mathcal{A}$, defined in \eqref{GimpliA}, fulfill (A1)--(A4).
Then,
\begin{equation}\label{spat}
\G_{\DD}(\S,\DD)(\G_{\S}(\S,\DD))^T \le 0 \qquad \textrm{ for all } (\S,\DD)\in \mathcal{A}.
\end{equation}
\end{lemma}

\begin{proof}[Proof of Lemma~\ref{GvsAAA}]
We provide here only a formal proof and avoid the use of $\varepsilon$-approximation (as this can be done as in the proof above, see \eqref{pepa998}). Thus, for simplicity, we assume that for any couple $(\S,\DD)\in \A$ one can write $\S$ as a Lipschitz mapping of $\DD$. Then, since for all $\DD$ we have that $\G(\S(\DD),\DD)=\0$, it follows that
$$
\0 = \G_{\S}(\S,\DD)\S_{\DD}(\DD) + \G_{\DD}(\S, \DD).
$$
Hence, by multiplying the result by $(\G_{\S}(\S,\DD))^T$, we see that
$$
-\G_{\DD}(\S, \DD)(\G_{\S}(\S,\DD))^T= \G_{\S}(\S,\DD)\S_{\DD}(\DD)(\G_{\S}(\S,\DD))^T.
$$
Thanks to the fact that $\A$ is monotone, we have that $\S_{\DD}\ge 0$ and consequently, for arbitrary $\Z\in \mathbb{R}^{N\times d}$,
$$
\begin{aligned}
-\G_{\DD}(\S, \DD)(\G_{\S}(\S,\DD))^T \Z \cdot \Z&= \G_{\S}(\S,\DD)\S_{\DD}(\DD)(\G_{\S}(\S,\DD))^T\Z \cdot \Z \\
&= \S_{\DD}(\DD)\left((\G_{\S}(\S,\DD))^T\Z\right) \cdot \left((\G_{\S}(\S,\DD))^T\Z\right) \ge 0,
\end{aligned}
$$
which finishes the proof.
\end{proof}
}

\section{Algebraic \texorpdfstring{$\e$}{e}-approximations of the graph \texorpdfstring{$\A$}{A}}\label{graph-eps}

In this section, we construct two different suitable $\varepsilon$-approximations of the maximal monotone $p$-coercive graph and show that these approximate graphs are Lipschitz continuous and uniformly monotone $2$-coercive graphs. Another advantage of these approximate graphs comes from their algebraic construction that is easy to incorporate into numerical schemes and their implementation. Finally, we study the convergence properties. In fact, the first approximation in Definition~\ref{construct} starts with the notion of maximal monotone $p$-coercive graph, while the second approximation in Definition~\ref{construct2} is directly linked to the implicit constitutive equation $\cG(\S,\DD) = \0$ with $\cG$ fulfilling (G1)-(G4).

{ The structure of this section is the following. We first define three approximations (see Definition~\ref{construct}, Definition~\ref{construct2}, and Remark~\ref{construct3}). Then, in Lemma~\ref{Agraf} and Lemma~\ref{Ggraf}, we study properties of these approximations and present an approach to verifying that the limit of the sequences arising from these $\varepsilon$-approximations fulfills the constitutive equation $\cG(\S,\DD) = \0$. After formulating these convergence lemmas, we comment on the novelties of these results. Finally, we prove first Lemma~\ref{Agraf} and then Lemma~\ref{Ggraf}.}


\begin{definition}[Construction of the approximate graphs] \label{construct}Let $\A$ be a maximal monotone $p$-coercive graph, see Definition~\ref{maxmongrafA}, and let $\e>0$. We define
\begin{subequations}\label{AAA}
\begin{align}
\Ae &:= \{(\tS,\tD) \in \R^{{N\times d}} \times \R^{{N\times d}}; \exists (\oS,\oD) \in \A, \tS = \oS, \tD = \oD + \e \oS\}, \label{Ae}\\
\Aee &:= \{(\S,\DD) \in \R^{{N\times d}} \times \R^{{N\times d}}; \exists (\tS,\tD) \in \Ae, \S = \tS + \e \tD, \DD = \tD\}. \label{Aee}
\end{align}Agraf
\end{subequations}
\end{definition}

\begin{remark}
There is no apparent reason for the lower and the upper index in the definition of the graph $\Aee$ to be the same. However, making them different (e.g., $\Ae^e$ for $\e, e>0$) would not bring any analytical advantage, generality, or simplicity.
\end{remark}
\begin{definition}[Construction of the approximation to the constitutive equations] \label{construct2}
Let $\cG$ satisfy (G1)--(G4) and let $\e>0$. We set $\cG_{\e}(\S,\DD):= \cG(\S-\e \DD,\DD-\e \S)$.
\end{definition}

\begin{remark} \label{construct3}
Instead of { the} null points of $\cG_{\e}$, one could also use an alternative approximation $\tilde{\cG}_{\e}(\S,\DD):= \cG(\S,\DD)\pm \e(\S - \DD)$ { where the positive sign is used if $\cG_{\S}(\S,\DD) \ge 0$ and the negative sign in the opposite case. The approximation $\tilde{\cG}_{\e}$} also leads to \underline{strictly} monotone and \underline{locally} Lipschitz graphs but we cannot guarantee $2$-coercivity and therefore the Hilbert structure of the approximation problem is lost. { That is the reason why we do not further study the convergence properties of the approximation $\tilde{\cG}_{\e}$.} On the other hand, $\tilde{\cG}_{\e}$ seems to be the easiest way of approximating the constitutive equation $\cG(\S,\DD)=\0$ { and we include it in the illustrative Figure~\ref{fig:compar}, where the three approximations $\Aee$, ${\cG}_{\e}$, and $\tilde{\cG}_{\e}$ are compared for two values of $\e$. For this comparison, we approximate the ``step" function on the left in Figure~\ref{figure1}, see also Example~\ref{Exsmykala} in Appendix A.}
\end{remark}

\begin{figure}
\begin{tabular}{ll}
\porovnanie{0.1}{30/7}       
\porovnanie{0.3}{30/7}       
\end{tabular}
\caption{Comparison of different approximations for $\e=0.1$ on the left and $\e=0.3$ on the right. The full line represents (for details see Example~\ref{Exsmykala}) $\g(\j,\dd)=\j - \dd - \tilde{a}(\sqrt{2}/2|\j+\dd|)(\j+\dd)=\0$, the dash dotted line $\Aee=\{(\j,\dd); \g(\j-\e \dd, (1+\e^2)\dd-\e\j)=\0\}$, the dotted line $\g_\e(\j,\dd)= \g(\j-\e \dd, \dd - \e \j)=\0$, and the dashed line stands for $\tilde{\g}_\e(\j, \dd)= \g(\j, \dd)+\e(\j-\dd)$.}
\label{fig:compar}
\end{figure}

For the approximations introduced in Definitions \ref{construct} and \ref{construct2} above, we establish the following results playing a key role in the subsequent analysis developed in this paper.
\begin{lemma}\label{Agraf}
Let $\A$ be a maximal monotone $p$-coercive graph. Then, for every $\e \in (0,1)$, $\Aee$ is a maximal monotone $2$-coercive graph. Moreover, there exists a unique single-valued mapping $\S_{\!\!\e}^*: \R^{{N\times d}} \to \R^{{N\times d}}$ satisfying
\begin{equation}\label{selection}
(\S, \DD) \in \Aee ~\Longleftrightarrow~ \S = \S_{\!\!\e}^*(\DD).
\end{equation}
In addition, $\S_{\!\!\e}^*(\0)=\0$ and $\S_{\!\!\e}^*$ is Lipschitz continuous and uniformly monotone, i.e. there exist $C_1{(\e)}$, $C_2{(\e)}>0$ such that, for all $\DD_1$, $\DD_2\in \R^{{N\times d}}$,
\begin{equation}
\begin{split}
|{\S_{\!\!\e}^*}(\DD_1)-{\S_{\!\!\e}^*}(\DD_2)|&\le C_2{(\e)}|\DD_1-\DD_2|, \\
({\S_{\!\!\e}^*}(\DD_1)-{\S_{\!\!\e}^*}(\DD_2)): (\DD_1-\DD_2)&\ge C_1{(\e)}|\DD_1-\DD_2|^2.
\end{split}
\label{defmonolip}
\end{equation}

Let, for any $U \subset Q$ bounded\footnote{{ The assumption $U \subset Q$ includes the possibility that $U=Q$, but it is useful from the point of view of applications to consider also the case when $U$ is a proper subset of $Q$.}} and measurable, $\Se, \De : U \to \R^{{N\times d}}$ be such that $(\Se, \De) \in \Aee$ a.e. in $U$ and there is a $C>0$ such that
\begin{equation}\label{intbound}
\int_U \Se : \De \ddd x  \ddd t \leq C, \qquad\qquad~\text{ uniformly with respect to } \e.
\end{equation}
Then, there exist $\S \in L^{p'}(U; \R^{{N\times d}})$, $\DD \in L^p(U; \R^{{N\times d}})$ so that (modulo subsequences)
\begin{equation}\label{min2r}
\begin{aligned}
\Se &\tow \S &&\text{ weakly in } L^{\min\{2,p'\}}(U; \R^{{N\times d}}), \\
\De &\tow \DD &&\text{ weakly in } L^{\min\{2,p\}}(U; \R^{{N\times d}}).
\end{aligned}
\end{equation}
Moreover, if
\begin{equation}\label{assomez}
\limsup_{\e \to 0_+} \int_U \Se : \De \ddd x  \ddd t \leq \int_U \S : \DD \d x \d t,
\end{equation}
then $(\S, \DD) \in \A$ almost everywhere in $U$ and,
\begin{equation}\label{limitanasobku}
\Se : \De \tow \S:\DD \quad \text{ weakly in } L^1(U).
\end{equation}
\end{lemma}
The next assertion concerns the properties of $\cG_{\e}$, see Definition \ref{construct2}.
\begin{lemma}\label{Ggraf}
{Let $\cG$ satisfy (G1)--(G4) with $p\in (1,\infty)$.  Then for every $\e \in (0,1)$, the null points of $\cG_{\e}$ generate a maximal monotone $2$-coercive graph $\Aeetilde$. Moreover, there exists a unique single-valued mapping $\S_{\!\!\e}^*: \R^{{N\times d}} \to \R^{{N\times d}}$ satisfying
\begin{equation}\label{selectione}
\cG_{\e}(\S,\DD) ~\Longleftrightarrow~ (\S, \DD) \in \Aeetilde ~\Longleftrightarrow~ \S = \S_{\!\!\e}^*(\DD).
\end{equation}
In addition, $\S_{\!\!\e}^*(\0)=\0$ and $\S_{\!\!\e}^*$ is Lipschitz continuous, uniformly monotone and satisfies \eqref{defmonolip}.

Furthermore, let, for any $U \subset Q$ bounded and measurable, $\Se, \De : U \to \R^{{N\times d}}$ be such that $(\Se, \De) \in \Aeetilde$ a.e. in $U$ and there is a $C>0$ such that
\begin{equation}\label{intbounde}
\int_U \Se : \De \ddd x  \ddd t \leq C, \qquad\qquad~\text{ uniformly with respect to } \e.
\end{equation}
Then, there exist $\S \in L^{p'}(U; \R^{{N\times d}})$, $\DD \in L^p(U; \R^{{N\times d}})$ such that \eqref{min2r} holds true. Moreover, \eqref{assomez} implies \eqref{limitanasobku} and  the limit satisfies $\cG(\S, \DD)=\0$ a.e. in $U$.
}
\end{lemma}
These results bring several novelties. First, we approximate, in a constructive way, a general maximal monotone $p$-coercive graph $\A$ by Lipschitz continuous and uniformly monotone $2$-coercive graphs $\Aee$ that can be identified with a single-valued (Lipschitz continuous and uniformly monotone) mapping. For such mappings there are many tools to obtain the existence of solution to the corresponding systems of PDEs. (We provide one such proof in Appendix~\ref{App3}.) Then, referring to the convergence parts of the above lemmas, we observe that to identify the limiting graph, it is just enough to check the validity of \eqref{assomez}. { The fact that the convergence parts of Lemma~\ref{Agraf} and Lemma~\ref{Ggraf} have a local character (i.e. the assumptions and the results hold on an arbitrary subset of $Q$) is important from the point of view of applications.} Note that several subtle tools have been developed to achieve \eqref{assomez} for various nonlinear problems of elliptic or parabolic type, mostly in fluid and solid mechanics, that can be used even if an energy equality is not available (or expressed differently, if the solution itself is not regular enough to be an admissible test function for the limiting problem). We refer to \cite{blechta2020} for details. { It is also worth noticing that our approximations are of a different nature than the commonly used ones. The way how we approximate the graph $\A$ in Definition~\ref{construct} shares certain similarities with the Yosida approximation or with an alternative approach developed in~\cite{FrMuTa04}. However, in our approach, the graph $\Aee$ is constructed completely explicitly as we ``stretch" and rotate the graph with respect to both axes, which is easy to implement, while the Yosida approximation requires the knowledge of the resolvent operator (which can be difficult to identify explicitly) and an alternative approach in \cite{FrMuTa04} is based on the a~priori knowledge of a certain Lipschitz function, which always exists but its explicit description may not be easy to identify (compare also with the proof of Lemma~\ref{onto}). The main advantage of our approach is in introducing the approximation $\cG_{\e}$, see Definition~\ref{construct2}. This is simple and does not require any additional knowledge of auxiliary functions or mappings.} Finally, we would like to point it out that  the approximation $\cG_{\e}$ might be more efficient when solving the problem~\eqref{problem} numerically, while the approximation $\Aee$ is easier to handle from the theoretical point of view and follows the classical  approaches in the theory of maximal monotone graphs.

\begin{proof}[Proof of Lemma~\ref{Agraf}] Throughout the proof of Lemma~\ref{Agraf}, we follow the notation indicated in the Definition~\ref{construct}, namely $(\S, \DD) \in \Aee$, $(\tS, \tD) \in \Ae$, and $(\oS, \oD) \in \A$, possibly with indices. The only exception is the limiting object defined in \eqref{min2r} as it is not a~priori defined to be in any of the graphs. We hope this notation may help to clarify the construction as well as the limiting procedure.

\smallskip
\paragraph{{\bf Step 1.} \underline{The existence of $\S^*_{\!\!\e}$}} In \eqref{multivalueD}, we identified the maximal monotone graph $\A$ with a possibly multivalued maximal monotone mapping $\DD^*$ defined on a subset of $\R^{{N\times d}}$. Thanks to Lemma~\ref{onto} we know that $\DD^* + \e \I$ is onto $\R^{{N\times d}}$ for any $\e \in (0,1]$. 
This surjectivity and the definition of $\DD^*$ then imply that for any $\tD \in \R^{{N\times d}}$ there is a couple $(\oS,\oD) \in \A$ such that
$\tD = \oD + \e \oS$. Setting simply $\tS:= \oS$, then
\begin{equation}\label{foranyDS}
\begin{aligned}
\text{for any $\tD \in \R^{{N\times d}}$ there exist $(\oS,\oD) \in \A$ and $\tS \in \R^{{N\times d}}$ such that }\\
\tS = \oS, \hspace{.6cm} \tD = \oD + \e \oS, \hspace{.3cm} \text{ and } \hspace{.3cm} (\tS,\tD) \in \Ae.\hspace{1.6cm}
\end{aligned}
\end{equation}
By definition of $\Aee$, for any $(\S, \DD) \in \Aee$ there exists a couple $(\tS,\tD) \in \Ae$ such that $\S = \tS + \e \tD, \DD = \tD$. However, thanks to \eqref{foranyDS}, we obtain that
\begin{equation*}
\text{for any $\DD \in \R^{{N\times d}}$, there exists a $\S \in \R^{{N\times d}}$ such that $(\S, \DD) \in \Aee$,}
\end{equation*}
which guarantees the existence {of} a mapping $\S_{\!\!\e}^*$ as defined in \eqref{selection}.

\smallskip
\paragraph{{\bf Step 2.} \underline{Properties of $\S_{\!\!\e}^*$ { and maximality of $\Aee$}}} To prove its properties, for $i=1,2$, let $(\S_i, \DD_i) \in \Aee$, $(\tS_i, \tD_i) \in \Ae$, and $(\oS_i, \oD_i) \in \A$, which relate to each other according to the definitions in \eqref{AAA}.
By means of the monotonicity (A2), we obtain
\begin{equation}\label{lipsch}
\begin{aligned}
(\tS_1 - \tS_2):(\tD_1 - \tD_2) &= (\oS_1 - \oS_2):(\oD_1 - \oD_2 + \e(\oS_1 - \oS_2)) \\
&\geq \e |\oS_1 - \oS_2|^2 = \e |\tS_1 - \tS_2|^2.
\end{aligned}
\end{equation}
{ This then implies that} 
\begin{align*}
(\S_1 - \S_2)&:(\DD_1 - \DD_2) = (\tS_1 - \tS_2 + \e (\tD_1 - \tD_2)):(\tD_1 - \tD_2) \\
&\geq \e |\tS_1 - \tS_2|^2 + \e |\tD_1 - \tD_2|^2 \\
&= \e |\S_1 - \S_2 - \e (\DD_1 - \DD_2)|^2 + \e |\DD_1 - \DD_2|^2 \\
&= \e \left(|\S_1 -\S_2 |^2 + (1+\e^2) |\DD_1 - \DD_2|^2  - 2\e (\S_1 -\S_2):(\DD_1 - \DD_2)\right),
\end{align*}
and consequently
\begin{equation*}
(\S_1 - \S_2):(\DD_1 - \DD_2)\geq \frac{\e}{1+2\e^2} \left(|\S_1 -\S_2 |^2 + (1+\e^2) |\DD_1 - \DD_2|^2 \right),
\end{equation*}
which proves the Lipschitz continuity and the uniform monotonicity of $\S_{\!\!\e}^*${, the latter implying that $\S_{\!\!\e}^*$ is a single-valued mapping.} It also gives the $2$-coercivity (A4) of $\Aee$ (and $\S_{\!\!\e}^*$ as well) by taking $(\S_2,\DD_2) = (\0,\0)$.

{ The maximality of $\Aee$ follows from the properties of the mapping $\S_{\!\!\e}^*$ by applying Minty's method. Indeed, let $(\S,\DD)\in \mathbb{R}^{N\times d} \times \mathbb{R}^{N\times d}$ be given such that for all $(\Se,\De)\in \Aee$ there holds
$$
(\S - \Se):(\DD - \De)\geq 0.
$$
Then, for arbitrary $\Z\in \mathbb{R}^{N\times d}$ and $\delta>0$, we can set $\De:=\DD - \delta \Z$ and $\Se:=\S_{\!\!\e}^*(\De)$. Obviously, $(\Se, \De)\in \Aee$. Consequently, the above inequality gives
$$
(\S - \S_{\!\!\e}^*(\DD - \delta \Z)):\Z\geq 0.
$$
Letting $\delta \to 0_+$ and using the continuity of  $\S_{\!\!\e}^*$ we conclude that
$$
(\S - \S_{\!\!\e}^*(\DD)):\Z\geq 0 \quad \text{ for all } \Z\in \mathbb{R}^{N\times d} \qquad  \implies \qquad \S = \S_{\!\!\e}^*(\DD).
$$
Consequently, $(\S,\DD)\in \Aee$.}

\paragraph{{\bf Step 3.} \underline{Proof of \eqref{min2r}}} Let $(\Se, \De) \in \Aee$ almost everywhere in $U$, and let $\int_U \Se : \De \d x {\d t}\leq C$. From the $2$-coercivity of the graph $\Aee$, we know that $\Se, \De \in L^2({U}; \R^{{N\times d}})$ and for
\begin{equation}\label{defto}
\Sno := \Se - \e \De, \hspace{1cm} \Dno:= \De - \e \Sno
\end{equation}
we have that $(\Sno, \De) \in \Ae$ almost everywhere in ${U}$ and $(\Sno, \Dno) \in \A$ almost everywhere in ${U}$. Thanks to the monotonicity (A2) {of $\A$},
\begin{equation}\label{ee22}
\Se \!:\! \De \!=\!(\Sno + \e \De)\!:\! \De \!=\! \e |\De|^2 + \Sno\!:\!\De\!=\! \e |\De|^2 + \e |\Sno|^2 + \Sno\!:\!\Dno \geq 0,
\end{equation}
but also
\begin{equation}\label{Acoerc}
\Se : \De \geq \e |\De|^2 + \e |\Sno|^2 + C_1 |\Sno|^{p'} + C_1 |\Dno|^{p} - C_2,
\end{equation}
due to the $p$-coercivity (A4) of $\A$. Therefore, using the assumption \eqref{intbound},
\begin{equation}\label{es}
\int_U \e |\De|^2 + \e |\Sno|^2 + |\Sno|^{p'} + |\Dno|^{p} \d x \d t \leq C, \qquad\qquad ~\text{ uniformly with respect to } \e.
\end{equation}
Using the definitions in \eqref{defto},
\begin{align*}
\int_U |\Se|^{\min\{2,p'\}} \d x \d t &= \int_U |\Sno + \e \De|^{\min\{2,p'\}} \d x \d t \leq C, \\
\int_U |\De|^{\min\{2,p\}} \d x \d t &= \int_U |\Dno + \e \Sno|^{\min\{2,p\}} \d x \d t\leq C,
\end{align*}
and due to reflexivity of $L^p(Q; \R^{{N\times d}})$ for any $p >1$, there exist $\S$, $\DD$, $\oS$, and $\oD$ such that
\begin{equation}\label{conver}
\begin{aligned}
\Se &\tow \S &&\text{weakly in } L^{\min\{2,p'\}}(U; \R^{{N\times d}}), \\
\De &\tow \DD &&\text{weakly in } L^{\min\{2,p\}}(U; \R^{{N\times d}}), \\
\Sno &\tow \oS &&\text{weakly in } L^{p'}(U; \R^{{N\times d}}), \\
\Dno &\tow \oD &&\text{weakly in } L^p(U; \R^{{N\times d}}).
\end{aligned}
\end{equation}
Next, we show that $\S = \oS$ almost everywhere in $U$ and $\DD = \oD$ almost everywhere in $U$. From \eqref{es} and \eqref{conver} we have
\begin{equation}\label{to0}
\begin{aligned}
\e \De &\tow \0 &&\text{ weakly in } L^2(U; \R^{{N\times d}}), \\
\e \Sno &\tow \0 &&\text{ weakly in } L^2(U; \R^{{N\times d}}),
\end{aligned}
\end{equation}
and also
\begin{equation}\label{rovne}
\begin{aligned}
\S \leftharpoonup \Se &= \Sno + \e \De \tow \oS &&\implies \hspace{1.07cm} \S = \oS\,\,\text{ in } L^{\min\{2,p'\}}(U; \R^{{N\times d}}),  \\
\DD \leftharpoonup \De &= \Dno + \e \Sno \tow \oD &&\implies \hspace{1cm} \DD = \oD\, \text{ in } L^{\min\{2,p\}}(U; \R^{{N\times d}}).
\end{aligned}
\end{equation}
Together, \eqref{conver} and \eqref{rovne} prove the statement \eqref{min2r}. { The convergence results \eqref{conver}--\eqref{rovne} hold for a properly chosen subsequence $\e_n \to 0_+$ that is from now on considered to be fixed.}

\smallskip

\paragraph{{\bf Step 4.} \underline{Proof of \eqref{limitanasobku} for $\overline{\S^n}$ and $\overline{\DD^n}$}}  {For an arbitrary sequence $(\S^{\e_n}, \DD^{\e_n})\in \A^{\e_n}_{\e_n}$ satisfying \eqref{intbound} and \eqref{min2r} as $\varepsilon_n \to 0_+$, we set $(\S^n,\DD^n):=(\S^{\e_n}, \DD^{\e_n})$ and $(\S^m,\DD^m):=(\S^{\e_m}, \DD^{\e_m})$. Then, using~\eqref{defto} as an inverse definition to~\eqref{AAA}, we define $(\overline{\S^{n}},\overline{\DD^{n}})$ and $(\overline{\S^{m}},\overline{\DD^{m}}) \in \A$. In this step, we prove~\eqref{oeweakly}, which is~\eqref{limitanasobku} for $(\overline{\S^{n}},\overline{\DD^{n}}) \in \A$, and in Step 5, we finish the proof of~\eqref{limitanasobku} for $(\S^{n}, \DD^{n})\in \A^{\e_n}_{\e_n}$.  } From the monotonicity of $\A$,
\begin{equation}\label{abshod}
|(\overline{\S^n} - \overline{\S^m}):(\overline{\DD^n} - \overline{\DD^m})| = (\overline{\S^n} - \overline{\S^m}):(\overline{\DD^n} - \overline{\DD^m}).
\end{equation}
Also, for any fixed $m,n$, one has that $\overline{\S^n},\overline{\S^m},\overline{\DD^n},\overline{\DD^m} \in L^2(Q; \R^{{N\times d}})$. Then, for any $U \subset Q$, using the weak convergence results \eqref{conver} and \eqref{rovne},
\begin{equation}\label{limm}
\begin{aligned}
\limsup_{m\to \infty} &\int_U (\overline{\S^n} - \overline{\S^m}):(\overline{\DD^n} - \overline{\DD^m}) \d x \d t \\
&= \limsup_{m\to \infty} \int_U \overline{\S^n} :(\overline{\DD^n} - \overline{\DD^m}) + \overline{\S^m}:(\overline{\DD^m}- \overline{\DD^n}) \d x \d t \\
&= \int_U \overline{\S^n}:(\overline{\DD^n} - \DD) \d x \d t - \int_U \S :\overline{\DD^n} \d x \d t + \limsup_{m \to \infty} \int_U \overline{\S^m}:\overline{\DD^m} \d x \d t
\end{aligned}
\end{equation}
and
\begin{equation}\label{limnm}
\begin{aligned}
\limsup_{n\to \infty} \limsup_{m\to \infty} &\int_U (\overline{\S^n} - \overline{\S^m}):(\overline{\DD^n} - \overline{\DD^m}) \d x \d t \\
&= \limsup_{n\to \infty} \int_U \overline{\S^n} :\overline{\DD^n} \d x \d t - \liminf_{n\to \infty} \int_U \overline{\S^n} :\DD \d x \d t \\
&\quad - \liminf_{n\to \infty} \int_U \S :\overline{\DD^n} \d x \d t + \limsup_{m\to \infty} \int_U \overline{\S^m}:\overline{\DD^m} \d x \d t\\
&= 2\left(\limsup_{n \to \infty} \int_U \overline{\S^n}:\overline{\DD^n} \d x \d t - \int_U \S :\DD \d x \d t   \right).
\end{aligned}
\end{equation}
However, using the definitions \eqref{defto} {and the computation~\eqref{ee22}}, we obtain the estimate
\begin{equation*}
{\overline{\S^n}:\overline{\DD^n} \leq \S^n : \DD^n,}
\end{equation*}
and if we combine it with the assumption \eqref{assomez}, we arrive at
\begin{equation}\label{useass}
\limsup_{n\to \infty} \int_U \overline{\S^n}:\overline{\DD^n} \d x \d t \leq  \limsup_{n\to \infty} \int_U \S^n:\DD^n \d x \d t \leq \int_U \S:\DD \d x \d t.
\end{equation}
Now, the results \eqref{abshod}, \eqref{limnm} and \eqref{useass} together imply that
\begin{equation}\label{abs0}
\lim_{n\to \infty} \lim_{m\to \infty} \int_U |(\overline{\S^n} - \overline{\S^m}):(\overline{\DD^n} - \overline{\DD^m})| \d x \d t =0,
\end{equation}
which proves that, for any $\varphi \in L^\infty(U)$,
\begin{equation*}
\lim_{n\to \infty} \lim_{m\to \infty} \int_U (\overline{\S^n} - \overline{\S^m}):(\overline{\DD^n} - \overline{\DD^m}) \varphi \d x \d t =0.
\end{equation*}
Using the boundedness of $\varphi$ and a procedure very similar to that in \eqref{limm} and \eqref{limnm} we observe that
\begin{align*}
0 &= \lim_{n\to \infty} \left( \int_U \overline{\S^n}:(\overline{\DD^n} -\DD)\varphi \d x \d t - \!\int_U\! \S :\overline{\DD^n} \varphi \d x \d t + \!\lim_{m\to \infty} \int_U \overline{\S^m}:\overline{\DD^m} \varphi \d x \d t \right)\\
&= 2\left(\lim_{n\to \infty} \int_U \overline{\S^n}:\overline{\DD^n} \varphi \d x \d t - \int_U \S :\DD \varphi \d x \d t \right).
\end{align*}
As this is true for any $\varphi \in L^\infty(U)$, we obtain
\begin{equation}\label{oeweakly}
\overline{\S^n}:\overline{\DD^n} \tow \S :\DD~ \text{ weakly in }~L^1(U).
\end{equation}

\smallskip

\paragraph{{\bf Step 5.} \underline{$(\S, \DD) \in \A$}} Let $\x$ be a Lebesgue point of $\S$, $\DD$, and $\S\!:\!\DD$. Let $(\oS, \oD) \in \A$ be arbitrary (independent of $\e$ and $\x$). Then, for any $\varphi \in L^\infty(U)$, $\varphi \geq 0$, using the monotonicity of $\A$ and the weak convergence results \eqref{conver}, \eqref{rovne}, and \eqref{oeweakly}, we have that
\begin{align*}
0 &\leq \lim_{\e_n \to 0_+} \int_U (\overline{\S^n} - \oS):(\overline{\DD^n} - \oD)\varphi \d x \d t \\
&=\lim_{\e_n \to 0_+} \int_U \overline{\S^n}:\overline{\DD^n} \varphi  -\overline{\S^n}: \oD \varphi - \oS:(\overline{\DD^n} - \oD) \varphi \d x \d t \\
&= \int_U (\S - \oS):(\DD - \oD) \varphi \d x \d t.
\end{align*}
Set $\varphi := \frac{1}{|B_\rho(\x)|}\chi_{B_\rho(\x)}$, and let $\rho \to 0_+$. Since $\x$ is a Lebesgue point,
\begin{equation*}
0\leq \lim_{\rho \to 0_+} \frac{1}{|B_\rho(\x)|} \int_{B_\rho(\x)} (\S - \oS):(\DD - \oD) \d x \d t = (\S(\x) - \oS):(\DD(\x) - \oD),
\end{equation*}
and this holds for any $(\oS, \oD) \in \A$; thus, by the maximality of $\A$, see (A3), we obtain that $(\S(\x),\DD(\x)) \in \A$.

Finally, we converge with $(\S^n,\DD^n)\in \A^{\e_n}_{\e_n}$, using the assumption \eqref{assomez} and the results \eqref{ee22} for the first, and \eqref{limnm} with \eqref{abs0} for the second equality,
\begin{align*}
\int_U \S:\DD \d x \d t  &\geq \limsup_{\e_n \to 0_+} \int_U \S^n:\DD^n \d x \d t \\
&= \limsup_{\e_n \to 0_+} \int_U  \overline{\S^n} : \overline{\DD^n} + \e_n |\DD^n|^2 + \e_n |\overline{\S^n}|^2 \d x \d t \\
& = \int_U \S:\DD \d x \d t + \limsup_{\e_n \to 0_+} \int_U \e_n |\DD^n|^2 + \e_n |\overline{\S^n}|^2 \d x \d t.
\end{align*}
Hence, the last integral vanishes as $\e_n \to 0_+$, and therefore $(\sqrt{\e_n}\, \DD^n)$ and $(\sqrt{\e_n}\,\overline{\S^n})$ converge strongly to zero in $L^2(U; \R^{{N\times d}})$, in contrast with the weak convergence result in \eqref{to0}.

Finally, since $(\S^n : \DD^n) = ( \overline{\S^n} : \overline{\DD^n} + \e_n|\overline{\S^n}|^2 + \e_n|\DD^n|^2)$, we use that the first term converges weakly in $L^1(U)$ to the desired limit thanks to~\eqref{oeweakly} and the last two terms converge strongly to zero in $L^1(U)$ to obtain the final statement \eqref{limitanasobku}.
\end{proof}

\begin{proof}[Proof of Lemma~\ref{Ggraf}] 
First, recalling that $\cG_{\e}(\S,\DD) = \cG(\S-\e\DD,\DD-\e\S)$, it is straightforward to observe that $(\bar{\S},\bar{\DD})$ is a null point of $\cG$ if and only if the couple $(\S,\DD)$ defined through
\begin{equation}\label{strik}
\DD= \frac{\bar{\DD}+\e \bar{\S}}{(1-\e^2)}, \qquad \S= \frac{\bar{\S}+\e \bar{\DD}}{(1-\e^2)}
\end{equation}
is a null point of $\cG_{\e}$. Then, since $\DD^* + \e \I$ is onto (see \eqref{multivalueD} for the definition of $\DD^*$ and Step 1 in proof of Lemma~\ref{onto}), we see that $\DD$ can be understood as a function of $\S$ and analogously (by interchanging the role of $\S$ and $\DD$) $\S$ can be understood as a function of $\DD$. Next, we show that these mappings are uniformly monotone and Lipschitz continuous, which implies that $\cG_{\e}$ generates a maximal monotone $2$-coercive graph. Indeed, let $(\S_1,\DD_1)$ and $(\S_2, \DD_2)$ be two null points of $\cG_{\e}$. Then, $(\S_i-\e\DD_i, \DD_i-\e\S_i)$ are null points of $\cG$ and, as the graph generated by $\cG$ is by Lemma \ref{GvsA} monotone, we have
$$
\begin{aligned}
0&\le ((\S_1-\e\DD_1) - (\S_2-\e\DD_2)): ((\DD_1-\e\S_1)-(\DD_2-\e\S_2))\\
&= ((\S_1-\S_2)-\e(\DD_1-\DD_2)): ((\DD_1-\DD_2)-\e(\S_1-\S_2))\\
&= (1+\e^2)(\S_1-\S_2):(\DD_1-\DD_2) -\e(|\DD_1-\DD_2|^2+ |\S_1-\S_2|^2).
\end{aligned}
$$
Consequently,
\begin{equation}\label{inimmm}
\frac{\e}{1+\e^2}(|\DD_1-\DD_2|^2+ |\S_1-\S_2|^2)\le (\S_1-\S_2):(\DD_1-\DD_2),
\end{equation}
which is the desired uniform monotonicity and which implies, after applying the Cauchy--Schwarz inequality to the right-hand side, the Lipschitz continuity. Consequently, the null points of $\cG_{\e}$ generate { a monotone $2$-coercive graph. The maximality then follows from Minty's method; compare also with Step~2 of the proof of Lemma~\ref{Agraf}.}

The rest of the proof coincides with the proof of Lemma~\ref{Agraf} with necessary minor changes due to a slightly different relation between the null points of $\cG$ and $\cG_{\e}$ given by \eqref{strik} and the relation between the graphs $\A$ and $\Aee$ given by Definition~\ref{construct}.
\end{proof}

\subsection*{Further auxiliary results}

We finish this section by stating three results. Two of them, Lemma \ref{aproxvA} and Lemma \ref{lemaodhmin}, will be needed in the proof of the main theorem. The third result, see Lemma \ref{pidilema}, is of independent interest within the context of earlier established results requiring a~priori the existence of a Borel measurable selection.

The first result establishes the condition that guarantees the stability of the graph $\A$ with respect to weakly converging sequences. It is a simpler variant of Lemma \ref{Agraf} above.
\begin{lemma}\label{aproxvA}
Let $\A$ be a maximal monotone $p$-coercive graph and let $U\subset{(0,T)\times} \R^d$ be a measurable bounded set. Assume that for every $n\in \mathbb{N}$, the mappings $\S^n, \DD^n : U \to \R^{{N\times d}}$ are such that $(\S^n, \DD^n) \in \A$ almost everywhere in $U$. In addition, let
\begin{equation*}
\int_U \S^n : \DD^n \ddd x \ddd t \leq C, \qquad\qquad ~\text{ uniformly with respect to}~ n \in \N.
\end{equation*}
Then, there exist $\S \in L^{p'}(U;\R{^{N\times d})}$ and $\DD\in L^p(U;\R^{{N\times d}})$ such that
\begin{align*}
\S^n &\tow \S &&\text{weakly in } L^{p'}(U; \R^{{N\times d}}), \\
\DD^n &\tow \DD &&\text{weakly in } L^{p}(U; \R^{{N\times d}}).
\end{align*}
Moreover, if
\begin{equation*}
\limsup_{n \to \infty} \int_U \S^n : \DD^n \ddd x \ddd t \leq \int_U \S : \DD \ddd x \ddd t,
\end{equation*}
then $(\S,\DD) \in \A$ almost everywhere in $U$ and $\S^n:\DD^n \tow \S:\DD$ weakly in $L^1(U)$.
\end{lemma}

\begin{proof}
See Lemma 1.2.2 in \cite{BGMS3} or Lemma \ref{Agraf} above.
\end{proof}

We also prove the uniform ($\e$-independent) coercivity estimate for $\Aee$.
\begin{lemma}\label{lemaodhmin}
There exist $\tilde{C_1}, \tilde{C_2} \in \R_+$ such that for all $\e\in (0,1)$ and all  $(\Se,\De)\in \Aee$ there holds
\begin{equation}\label{odhadAr2}
\Se : \De \geq \tilde{C_1}(|\Se|^{\min\{p', 2\}}) + |\De|^{\min\{p, 2\}}) - \tilde{C_2}.
\end{equation}
\end{lemma}
\begin{proof}
Let $(\S,\DD)\in \A$ be the couple corresponding to $(\Se,\De)\in \Aee$ according to Definition~\ref{construct}. Then
\begin{equation*}
\Se = \S + \e \De, ~\text{ and }~ \De = \DD+\e\S.
\end{equation*}
Now, using the $p$-coercivity of $\A$, we get \eqref{Acoerc}, and if we compute
\begin{align*}
|\De|^{\min\{p,2\}} &= |\DD + \e \S|^{\min\{p,2\}}\leq C(|\DD|^p + \e |\S|^2 +1), \\
|\Se|^{\min\{p',2\}} &= |\S + \e \De|^{\min\{p',2\}}\leq C(|\S|^{p'} + \e |\De|^2+1),
\end{align*}
and combine these, we obtain
\begin{equation*}
|\Se|^{\min\{p',2\}} + |\De|^{\min\{p,2\}} \leq C(|\DD|^p + \e |\S|^2 +|\S|^{p'} + \e |\De|^2+1) \leq C(\Se : \De +1).
\end{equation*}
\end{proof}

The next lemma is of interest within the context of mathematical methods for general constitutive equations of the form $\cG(\S,\DD)=\0$ (associated with the graph $\A$) developed earlier for fluid flow problems, see~\cite{BGMS1,BGMS2,BGMS3,BMZ}. In these studies, the assumption on the existence of a Borel measurable selection played an important role both for constructing an approximating single-valued {mapping} (by convolution) and for showing that
\begin{equation}
  \textrm{ for each } \DD\in L^p \textrm{ there is } \S\in L^{p'} \textrm{ such that } (\S,\DD) \in \A. \label{star}
\end{equation}
In this study, we do not require the existence of a Borel measurable selection due to a different approximation scheme developed above in this section. For the sake of completeness, we also show that the property \eqref{star} is available.


\begin{lemma}\label{pidilema}
Let $p\in(1,\infty)$ and let $\A$ be a maximal monotone $p$-coercive graph. Then, for every $\DD \in L^p(Q; \R^{{N\times d}})$, there exists $\S \in L^{p'}(Q; \R^{{N\times d}})$ such that $(\S, \DD) \in \A$ almost everywhere in $Q$.
\end{lemma}
\begin{proof}
For $k \in \N$, define $\DD_k := \DD \chi_{\{|\DD|\leq k\}}$. Recall the definition of $\Aee$ \eqref{Aee} and its selection $\S^*_{\!\!\e}$ \eqref{selection}. Then, by definition of selection, $(\S^*_{\!\!\e}(\DD_k), \DD_k) \in \Aee$ almost everywhere in $Q$. 
{We can moreover apply Young's inequality to the left-hand side of~\eqref{odhadAr2} to obtain, for small $\delta>0$,
\begin{equation*}
|\S^*_{\!\!\e}(\DD_k)|^{\min\{p', 2\}}\leq C \S^*_{\!\!\e}(\DD_k) : \DD_k \leq C(\delta |\S^*_{\!\!\e}(\DD_k)|^{\min\{p', 2\}} + C(\delta)|\DD_k|^{\max\{p, 2\}}),
\end{equation*}
which implies that
\begin{equation*}
|\S^*_{\!\!\e}(\DD_k)| \leq C|\DD_k|^{\frac{\max\{p,2\}}{\min\{p',2\}}}\leq C |\DD_k|^{\max\{p-1,1\}} \leq C k^{\max\{p-1,1\}}.
\end{equation*}}
Then, there exists a $\S_k$ such that, as $\e \to 0_+$,
\begin{equation*}
\S^*_{\!\!\e}(\DD_k) \tow^* \S_k \text{ weakly$^*$ in } L^\infty(Q; \R^{{N\times d}}).
\end{equation*}
Then we have the limit
\begin{equation*}
\lim_{\e \to 0_+} \iq \S^*_{\!\!\e}(\DD_k): \DD_k \d x \d t = \iq \S_k : \DD_k \d x \d t,
\end{equation*}
and thanks to  Lemma~\ref{aproxvA} we know that $(\S_k, \DD_k) \in \A$ almost everywhere in $Q$. Therefore,
\begin{equation*}
C_1 (|\S_k|^{p'} + |\DD_k|^p) - C_2 \leq \S_k : \DD_k \leq \frac{C_1}{p'} |\S_k|^{p'} + C|\DD_k|^p,
\end{equation*}
and then
\begin{equation*}
\iq |\S_k|^{p'} + |\DD_k|^p \d x \d t \leq C \iq |\DD_k|^p \d x \d t \leq  \iq |\DD|^p \d x \d t \leq C,
\end{equation*}
where the boundedness follows from the assumption. Finally, as $k \to +\infty$, we have for subsequences that
\begin{align*}
\S_k &\tow \S &&\text{weakly in }  L^{p'}(Q; \R^{{N\times d}}), \\
\DD_k &\to \DD &&\text{strongly in } L^{p}(Q; \R^{{N\times d}}),
\end{align*}
so $\lim_{k \to \infty} \iq \S_k : \DD_k \d x \d t = \iq \S : \DD \d x \d t$, which finishes the proof by use of Lemma \ref{aproxvA}.
\end{proof}

\section{Proof of Theorem~\ref{result}}\label{proofs}
The proof is based on the identification of the null set of $\G$ with a maximal monotone $p$-coercive graph $\A$ and on its subsequent approximation by the Lipschitz continuous and uniformly monotone $2$-coercive graphs $\Aee$ constructed and analyzed in Section~\ref{graph-eps}. The solution of the problem is then obtained by a limiting process as $\e\to 0_+$. In order to link the original $p$-coercive graph with the approximating $2$-coercive graphs, we need to consider a smoother right-hand side $\f$. More precisely, we define
\begin{equation}\label{munu}
\begin{aligned}
&\mu := \min\{p,2\},  &&\mu' := \max\{p',2\}, \\
&\nu := \min\{p',2\}, &&\nu' := \max\{p,2\}.
\end{aligned}
\end{equation}
and then, in the first seven steps of the proof, we prove Theorem \ref{result} for $\f \in L^{\mu'}(0,T;V_{\mu}^*)$. In the final Step 8, once having a solution for such $\f$, we consider a sequence of solutions $\{(\u^m, \S^m)\}_{m \in \N}$ of the problem \eqref{problem}
in the sense of Theorem \ref{result} with the right-hand side $\{\f^m\}_{m \in \N} \subset L^{\mu'}(0,T;V_{\mu}^*)$ satisfying\footnote{Since $V^*_{\mu}$ is dense in $V_p^*$ for $\mu'\ge p'$, such sequence surely exists.} $\f^m \to \f$ in $L^{p'}(0,T;V_p^*)$ and we briefly comment why the weak limits $(\u,\S)$ of suitable subsequences $\{(\u^m, \S^m)\}_{m \in \N}$ solve the problem \eqref{problem} with the right-hand side $\f$.

\paragraph{{\bf Step 1.} \underline{Approximations}}
First, we introduce a graph $\A$ by
$$
\mathcal{A}:=\{(\S,\DD){;} \; \G(\S,\DD)=\0\}.
$$
Then, due to Lemma~\ref{GvsA}, it follows from the assumptions (G1)--(G4) that $\A$ is a maximal monotone $p$-coercive graph, i.e. $\A$ satisfies (A1)--(A4) in Definition~\ref{maxmongrafA}. Consequently, for an arbitrary $\e \in (0,1)$, we use {Definition~\ref{construct}} to construct $\e$-approximate graphs $\Aee$. Then, due to Lemma~\ref{Agraf}, we observe that $\Aee$ can be identified with a Lipschitz continuous and uniformly monotone single-valued mapping $\S^*_{\!\!\e}$ so that
$$
(\S,\DD)\in \Aee \Longleftrightarrow \S=\S^*_{\!\!\e}(\DD).
$$
Consequently, for every $\e \in (0,1)$, we can apply Lemma~\ref{S2thm} and find
\begin{equation*}
(\ve, \Se) \in \left(L^2(0,T; V) \cap \C([0,T];H)\right) \times L^2(Q; \R^{{N\times d}})
\end{equation*}
satisfying\footnote{{The term on the right-hand side of \eqref{WFS2e} can be also written as $\langle \f,  \vp \rangle_{V_{\mu}}$ due to the fact that $\mu \le 2$. }}
\begin{align}\label{WFS2e}
\langle \pt \ve,  \vp \rangle_V + \io \Se : \nabla \vp \d x &= \langle \f,  \vp \rangle_V \quad \textrm{ for a.a. } t\in (0,T] \textrm{ and for any } \vp \in V, \\
\label{pepa_cr2}\Se & =\S^*_{\!\!\e}(\nabla \ve) \quad \textrm{almost everywhere in } Q, \\
\label{ICS2e}
\lim_{t\to 0_+} \|\ve(t) - \u_0\|_H &= 0.
\end{align}

\paragraph{{\bf Step 2.} \underline{Uniform a~priori estimates}}


We set $\vp:= \ve$ in \eqref{WFS2e}, integrate over $(0,t)$, use that $\pt \ve \in L^2(0,T;V^*)$ and properties of the Gelfand triple \eqref{gelfand}, and obtain
\begin{equation*}
\frac12 \|\ve(t)\|^2_H + \iqt \Se : \Dve \d x \d \tau = \it \langle \f,  \ve \rangle_{V_\mu} \d \tau + \frac12 \|\u_0\|_H^2.
\end{equation*}
Using the estimate \eqref{odhadAr2} from Lemma~\ref{lemaodhmin}, we get
\begin{equation}\label{wfve}
\begin{aligned}
\frac12 \|\ve(t)\|^2_H + \tilde{C_1} \iqt |\Se|^{\nu} + |\Dve|^{\mu} \d x \d \tau &\leq \frac12 \|\ve(t)\|^2_H + \iqt \Se : \Dve \d x \d \tau +C \\
&\leq \it \langle \f,  \ve \rangle_{V_\mu} \d \tau + \frac12 \|\u_0\|_H^2 +C.
\end{aligned}
\end{equation}
Next, recalling the definition of the $V_\mu$-norm and using Young's inequality, we get
\begin{equation*}
\begin{aligned}
\langle \f,  \ve \rangle_{V_\mu} &\leq \|\f\|_{V_\mu^*} \left(\|\ve\|_H + \|\Dve \|_{L^\mu(\o)}\right) \\
&\leq \frac{\tilde{C_1}}{2} \|\Dve \|_{L^\mu(\o)}^\mu + C\left(\|\f\|_{V_\mu^*}^{\mu'} + (\|\ve\|_H^2 +1) \|\f\|_{V_\mu^*} \right).
\end{aligned}
\end{equation*}
Inserting this into \eqref{wfve}, using the assumptions on the data $\f$ and $\u_0$ and applying then Gronwall's lemma, we get
\begin{equation}\label{venekH}
\sup_{t \in (0,T)} \|\ve(t)\|_H \leq C, \qquad\qquad~\text{uniformly with respect to } \e \in (0,1).
\end{equation}
Referring again to \eqref{wfve} we then also conclude that
\begin{equation}\label{unif}
\begin{aligned}
\sup_{t \in (0,T)} \|\ve(t)\|^2_H &+ \iq |\Se|^{\nu} + |\Dve|^{\mu} \d x \d \tau \leq C, \qquad\qquad~\text{uniformly with respect to } \e \in (0,1).
\end{aligned}
\end{equation}
Moreover, we also have
\begin{equation}\label{SeDveC}
\iq \Se : \Dve \d x \d \tau \leq C, \qquad\qquad ~\text{uniformly with respect to } \e \in (0,1).
\end{equation}
Finally, note that \eqref{unif} also implies that
\begin{equation}\label{vemumu}
\|\ve\|_{L^\mu(0,T;V_{\mu})} \leq C, \qquad\qquad ~\text{uniformly with respect to } \e \in (0,1).
\end{equation}
To estimate the time derivative, denote $\du := \{ \w \in V_p \cap V; \|\w\|_{V_{\nu'}} \leq 1\}$. Note that $\du \subset V$, then we can set $\vp := \w \in \du$ in the equation \eqref{WFS2} to get the following
\begin{align*}
\|\pt \ve\|_{V_{\nu'}^*} &= \sup_{\du} \langle \pt \ve, \w \rangle_{V_{\nu'}} = \sup_{\du} \left( -\io \Se : \nabla \w \d x + \langle \f,  \w \rangle_{V_{\mu}} \right) \\
&\leq \sup_{\du} \left( \|\Se\|_{L{^{\nu}}(\Omega; \R^{{N\times d}})} \|\nabla \w\|_{L{^{\nu'}}(\Omega; \R^{{N\times d}})} + \|\f\|_{V_{\mu}^*} \|\w\|_{V_{\mu}} \right).
\end{align*}
Using the fact that ${V_{\nu'} \hookrightarrow V_{\mu}}$, taking the $\nu$-th power and integrating the result over $(0,T)$ we obtain, using also \eqref{unif},
\begin{equation}\label{unifpt}
\begin{aligned}
\iT \|\pt \ve\|^\nu_{V_{\nu'}^*} \d t &\leq \iT \|\Se\|^\nu_{L{^{\nu}}(\Omega; \R^{{N\times d}})} + \|\f\|^\nu_{V_\mu^*} \d t \leq C, \qquad\qquad  ~\text{uniformly with respect to } \e \in (0,1).
\end{aligned}
\end{equation}

\paragraph{{\bf Step 3.} \underline{Limit $\e\to 0_+$}}

Using \eqref{vemumu}, \eqref{unif}, \eqref{venekH} and \eqref{unifpt}, we obtain that, as $\e \to 0_+$,
\begin{equation}\label{converg}
\begin{aligned}
\ve &\tow \u &&\text{weakly in } L^{\mu}(0,T;V_{\mu}), \\
\Se &\tow \S &&\text{weakly in } L^{\nu}(Q; \R^{{N\times d}}), \\
\ve &\tow^* \u &&\text{weakly$^*$ in } L^{\infty}(0,T;H), \\
\pt \ve &\tow \pt \u &&\text{weakly in } L^{\nu}(0,T;V_{\nu'}^*).
\end{aligned}
\end{equation}
Moreover, \eqref{SeDveC} in combination with the result of the Lemma \ref{Agraf} gives
\begin{equation}\label{SerDer}
\S \in L^{p'}(Q; \R^{{N\times d}}) ~\text{ and }~ \D \in L^p(Q; \R^{{N\times d}}).
\end{equation}
{The latter, in combination with $\u\in L^{\mu}(0,T;V_{\mu})$, implies that $\u\in L^p(0,T;V_p)$.}

Next, take $\w \in {V_{\nu'}(\hookrightarrow V_{\mu})}$ and $\xi \in L^\infty(0,T)$ arbitrary. Setting $\vp := \xi \w$ in \eqref{WFS2e}, integrating the result over $(0,T)$, we obtain
\begin{equation*}
\iT \langle \pt \ve,  \xi \w \rangle_{V_{\nu'}} \d t + \iq \Se : \nabla \w \xi \d x \d t = \iT \langle \f,\xi \w \rangle_{V_{\mu}} \d t.
\end{equation*}
Noticing that all terms are well-defined, we can take the limit $\e \to 0_+$ and, by means of \eqref{converg}, we end up with
\begin{equation*}
\iT \langle \pt \u,  \xi \w \rangle_{V_{\nu'}} \d t + \iq \S : \nabla \w \xi \d x \d t= \iT \langle \f,\xi \w \rangle_{V_{\mu}} \d t.
\end{equation*}
Since this holds for all $\xi$, it implies that
\begin{equation}\label{slabaw}
\langle \pt \u, \w \rangle_{V_{\nu'}} + \io \S : \nabla \w \d x = \langle \f,\w \rangle_{V_{\nu'}} \quad \textrm{ for a.a. } t\in(0,T) \textrm{ and for all } {\w \in V_{\mu}}.
\end{equation}
 To verify \eqref{WFSr}, we need to show that \eqref{slabaw} holds true for all $\w\in V_p$. For this purpose, we need to improve the information about the time derivative.

\paragraph{{\bf Step 4.} \underline{Improved information regarding $\pt \u$}}

Thanks to the dense embedding $V_p \cap V \hookrightarrow V_p$, we can use \eqref{slabaw} for $\mathcal{W}_p:= \{ \w \in V_p \cap V; \|\w\|_{V_p} \leq 1\}$ as follows:
\begin{align*}
\|\pt \u\|_{V_p^*} &= \sup_{\mathcal{W}_p} \, \langle \pt \u, \w \rangle_{V_p} = \sup_{\mathcal{W}_p} \left( -\io \S : \nabla \w \d x + \langle \f,  \w \rangle_{V_p} \right) \\
&\leq \sup_{\mathcal{W}_p} \left( \|\S\|_{L^{p'}(\Omega; \R^{{N\times d}})} \|\nabla \w\|_{L^{p}(\Omega; \R^{{N\times d}})} + \|\f\|_{V_p^*} \|\w\|_{V_p} \right) \leq \|\S\|_{L^{p'}(\Omega; \R^{{N\times d}})} + \|\f\|_{V_p^*}.
\end{align*}
Applying the power $p'$, integrating over time $t \in (0,T)$, and using the results in \eqref{converg}, we obtain
\begin{equation*}
\iT \|\pt \u\|^{p'}_{V_p^*} \d t \leq \iT \|\S\|^{p'}_{L^{p'}(\Omega; \R^{{N\times d}})} + \|\f\|^{p'}_{V_p^*}\d t \leq C,
\end{equation*}
and again using the density of $V_p \cap V \hookrightarrow V_p$, we conclude that \eqref{slabaw} is valid for any $\w \in V_p$ and for almost every $t \in(0,T)$.

Moreover, thanks to $\u \in L^p(0,T;V_p)$, $\pt \u \in L^{p'}(0,T;V_p^*)$, and the Gelfand triple \eqref{gelfand}, there holds $\u \in \C([0,T];H)$.

\paragraph{{\bf Step 5.} \underline{Attainment of the initial datum}}

For $0<\epsilon\ll 1$ and $t\in (0, T-\epsilon)$, we first introduce a cut-off function $\eta \in \C^{0,1}([0,T])$ as a piecewise linear function consisting of three pieces:
\begin{equation}\label{eta}
\eta(\tau) =
\begin{cases}
1 &\text{if } \tau \in [0,t),   \\
1 + \frac{t-\tau}{\epsilon} &\text{if } \tau \in [t,t+\epsilon),\\
0 &\text{if } \tau \in [t+\epsilon, T].
\end{cases}
\end{equation}
Next, for $\w \in {V_{\nu'}}$, we set $\vp:= \eta \w$ in \eqref{WFS2} and integrate over $(0,T)$ to deduce that
\begin{equation*}
\begin{aligned}
\frac{1}{\epsilon} \int_t^{t+\epsilon} (\ve(\tau) , \w )_H \d x \d \tau + \int_{Q_{t+\epsilon}} \Se:\nabla \w \eta \d x \d \tau = \int_0^{t+\epsilon} \langle \f,\w \eta \rangle_{V_{\nu'}} \d \tau + (\u_0 , \w )_H.
\end{aligned}
\end{equation*}
Letting $\e \to 0_+$ and using the results established in \eqref{converg}, we conclude that
\begin{equation*}
\begin{aligned}
\frac{1}{\epsilon} \int_t^{t+\epsilon} (\u(\tau) , \w )_H \d x \d \tau &+ \int_{Q_{t+\epsilon}} \S:\nabla \w \eta \d x \d \tau = \int_0^{t+\epsilon} \langle \f,\w \eta \rangle_{V_{\nu'}} \d \tau + (\u_0 , \w )_H.
\end{aligned}
\end{equation*}
Since $\u \in \C([0,T];H)$, we can also take the limit $\epsilon \to 0_+$; hence
\begin{equation*}
(\u(t) , \w )_H + \iqt \S:\nabla \w \d x \d \tau  = \it \langle \f,\w \rangle_{V_{\nu'}} \d \tau + (\u_0 , \w )_H,
\end{equation*}
and finally, we let $t \to 0_+$ to get
\begin{equation*}
\lim_{t \to 0_+}(\u(t) , \w )_H =(\u_0 , \w )_H.
\end{equation*}
As $\w \in V_p \cap V$ was arbitrary and  $V_p \cap V$ is dense in $H$, we obtain that $\u(t) \tow \u_0$ weakly in $H$, but thanks to the continuity of $\u$ in $H$ we obtain the strong convergence \eqref{ICSr}.

\paragraph{{\bf Step 6.} \underline{Attainment of the constitutive equation}} The aim is to show that $(\S,\nabla \u)\in \A$ almost everywhere in $Q$, which is equivalent to showing $\G(\S,\nabla \u)=\0$ almost everywhere in $Q$. To this end, we need to verify the assumption \eqref{assomez} of Lemma \ref{Agraf}, i.e. we need to prove that, for all $t \in (0,T)$,
\begin{equation}\label{goal1}
\limsup_{\e \to 0_+} \iqt \Se : \Dve \d x \d \tau \leq \iqt \S : \D \d x \d \tau.
\end{equation}
Indeed, having \eqref{goal1}, Lemma \ref{Agraf} implies that $(\S, \D) \in \A$ almost everywhere in $Q_t$ and that $\Se : \De \tow \S:\D$ weakly in $L^1(Q_t)$. Thus, we have obtained the desired result on $Q_t$ for every $t \in (0,T)$, and therefore also on $Q$.

The relation \eqref{goal1} is achieved by the standard energy and weak lower semicontinuity  techniques used in parabolic systems and for the sake of completeness, we provide the proof also here.
In \eqref{WFS2e}, we set $\vp=\ve$ and integrate the result over $(0,t)$ for $t\in(0,T)$. We obtain
\begin{equation*}
\iqt \Se : \Dve \d x \d \tau = \it \langle \f,  \ve \rangle_{V_{\mu}} \d \tau + \frac12 \|\u_0\|_H^2- \frac12 \|\ve(t)\|^2_H .
\end{equation*}
Applying then the limes superior as {$\e \to 0_+$} and using the weak convergence of $\ve$ in $L^{\mu}(0,T;V_{\mu})$ we conclude that
\begin{equation}\label{slabave}
\begin{aligned}
\limsup_{\e \to 0_+} &\iqt \Se : \Dve \d x \d \tau = \it \langle \f,  \u \rangle_{V_\mu} \d \tau + \frac12 \|\u_0\|_H^2- \frac12 \liminf_{\e \to 0_+} \|\ve(t)\|^2_H.
\end{aligned}
\end{equation}
On the other hand, setting $\w=\u$ in \eqref{slabaw} (we already have the appropriate duality pairings to do so) and integrating it over $(0,t)$ we arrive at
\begin{equation}\label{slabav}
\iqt \S : \D \d x \d \tau = \it \langle \f,\u \rangle_{V_\mu} \d \tau + \frac12 \|\u_0\|_H^2- \frac12 \|\u(t)\|^2_H.
\end{equation}
Subtracting \eqref{slabav} from \eqref{slabave} gives
\begin{equation}\label{almostthere}
\begin{aligned}
\limsup_{\e \to 0_+} \iqt \Se : \Dve \d x \d \tau = \frac12 \|\u(t)\|^2_H- \frac12 \liminf_{\e \to 0_+} \|\ve(t)\|^2_H + \iqt \S : \D \d x \d \tau.
\end{aligned}
\end{equation}
That is, to verify \eqref{goal1}, it remains to show that
\begin{equation}\label{remain}
\|\u(t)\|^2_H \leq \liminf_{\e \to 0_+} \|\ve(t)\|^2_H.
\end{equation}
In case $V_p$ is compactly embedded into $H$ (i.e. if $p>2d/(d+2)$), the above relation is for a.a. $t\in (0,T)$ a consequence of the convergence results \eqref{converg} and the Aubin--Lions compactness lemma. Therefore, if $p>2d/(d+2)$,  \eqref{goal1} holds for almost all time, which is sufficient for finishing the proof. Nevertheless, in case we do not have $V_p$ compactly embedded into $H$, we proceed slightly differently and moreover, we obtain \eqref{remain} for all $t\in (0,T)$ (instead of for almost all $t$).

Let $0<\delta  \ll T$ and take $\vp= \ve$ in \eqref{WFS2e}. Integrating the result over $(t, t+\delta)$ and applying then integration by parts to the first term, we obtain
\begin{equation*}
\begin{aligned}
\frac12 \|\ve(t+\delta)\|_H^2+ \int_t^{t+\delta} \io \Se : \Dve \d x \d \tau =\int_t^{t+\delta} \langle \f,  \ve \rangle_{V_\mu} \d \tau+ \frac12 \|\ve(t)\|_H^2.
\end{aligned}
\end{equation*}
As  $\Aee$ is monotone, we observe that $\Se : \Dve\ge 0$ and we neglect the corresponding term resulting in an inequality. Integrating it with respect to $\delta$ over  $(0,\gamma)$ for $0<\gamma \ll 1$, we arrive at
\begin{equation*}
\frac12 \int_0^\gamma \|\ve(t+\delta)\|_H^2 \d \delta - \int_0^\gamma \int_t^{t+\delta} \langle \f,  \ve \rangle_{V_\mu} \d \tau \d \delta \leq \frac{\gamma}{2}  \|\ve(t)\|_H^2.
\end{equation*}
Taking the limes inferior as $\e \to 0_+$ and using, on the left-hand side, the established weak convergence for $\ve$ and the weak lower semicontinuity of the norm, followed by multiplication of the resulting  inequality by $\frac{2}{\gamma}$, we get
\begin{equation*}
\frac{1}{\gamma} \int_0^\gamma \|\u(t+\delta)\|_H^2 \d \delta - \frac{2}{\gamma} \int_0^\gamma \int_t^{t+\delta} \langle \f,  \u \rangle_{V_\mu} \d \tau \d \delta \leq \liminf_{\e \to 0_+} \|\ve(t)\|_H^2.
\end{equation*}
Finally, letting $\gamma \to 0_+$, using the continuity of $\u$ in $H$ and the fact that the duality between $\f$ and $\u$ is well-defined, we obtain \eqref{remain}.

\paragraph{{\bf Step 7.} \underline{Uniqueness of $\u$}}

Let $(\u_1, \S_1)$ and $(\u_2, \S_2)$ be two solutions to the problem \eqref{problem}. If we subtract their weak formulations, we obtain
\begin{equation*}
\langle \pt (\u_1 - \u_2),  \vp \rangle_{V_p} + \io (\S_1 - \S_2) : \nabla\vp \d x = 0.
\end{equation*}
Next, we set $\vp := (\u_1 - \u_2)$ to get
\begin{equation*}
\frac12 \dt \|\u_1 - \u_2\|^2_{V_p} + \io (\S_1 - \S_2) : (\D_1 - \D_2) \d x = 0,
\end{equation*}
however, due to the monotonicity of the graph $\A$, we obtain that each term is equal to zero. Finally, after integration over time $(0,t)$ for every $t \in (0,T)$, we use that both solutions satisfy the same initial condition and conclude that $\u_1(t) = \u_2(t)$ in $V_p$ for every $t \in (0,T)$.

\paragraph{{\bf Step 8.} \underline{Sketch of the proof of Theorem \ref{result} for $f\in L^{p'}(0,T; V_p^*)$}}

Since $V^*_{\mu}$ is dense in $V_p^*$ for $\mu'\ge p'$, for a given $f\in L^{p'}(0,T; V_p^*)$ there exists a sequence $\{\f^m\}_{m \in \N} \subset L^{\mu'}(0,T;V_{\mu}^*)$ satisfying
$$
\f^m \to \f \textrm{ in } L^{p'}(0,T;V_p^*).
$$
For each $m\in\N$ we consider a solution $(\u^m, \S^m)$ of the problem~\eqref{problem}
in the sense of Theorem~\ref{result} with the right-hand side~$\f^m$. Then, we proceed as in Steps~2--7, i.e. we derive uniform estimates for $\{(\u^m, \S^m)\}_{m \in \N}$, find appropriate weak limits~$(\u,\S)$, and study the limit as~$m\to \infty$. This is all done in the same way as (or in a slightly simpler way than) above. In particular, we use Lemma~\ref{aproxvA} for the verification that~the couple $(\S, \nabla \u)$ belongs to~$\A$. Note that~$\A$ remains unchanged throughout this step.

\begin{appendix}

\section{Prototypical examples}\label{App1}

With the aim to clarify the conditions (g1)--(g4) formulated in the introductory section and to fix the notation involved in their descriptions, we consider five examples of the implicit constitutive equations $\g(\j,\dd)=\0$ and show that they satisfy (g1)--(g4).

\begin{example}\label{Ex1}
The linear case $\j= \dd$, i.e.,
$$
\g(\j,\dd)=\j-\dd.
$$
\end{example}
\begin{proof}[Validity of (g1)--(g4) for Example~\ref{Ex1}]

To show that Example~\ref{Ex1} satisfies (g1)--(g4), we first notice that $\g(\0,\0)=\0$, $\g_{\j}(\j,\dd)=\I$, $\g_{\dd}(\j,\dd)=-\I$, $\g_{\j}(\j,\dd) - \g_{\dd}(\j,\dd)= 2\I$ and, by a simple computation,  $\g_{\dd}(\j,\dd) (\g_{\j}(\j,\dd))^T=-\I$, and therefore (g1) and~(g2) obviously  hold. Furthermore,
$$
\g(\j,\dd) \cdot \j = |\j|^2 - \j\cdot\dd \quad \textrm{ and } \quad \g(\j,\dd) \cdot \dd = \j\cdot\dd - |\dd|^2.
$$
Consequently, for a fixed $\dd\in\R^d$,
$$
\lim_{|\j|\to \infty} \g(\j,\dd)\cdot \j = \infty
$$
and, for a fixed $\j\in\R^d$,
$$
\lim_{|\dd|\to \infty} \g(\j,\dd)\cdot \dd = - \infty,
$$
which proves (g3). Finally,
$$
  \j\cdot\dd = \frac12(\j\cdot\dd + \j\cdot\dd) = \frac12|\j|^2 + \frac12|\dd|^2,
$$
where, in the last equality, we inserted first $\dd$ for $\j$ and then $\j$ for $\dd$ (using $\0 = \g(\j,\dd) = \j-\dd$). Hence, (g4) holds.

Note that setting $\g(\j,\dd) = \dd-\j$ leads to sign changes in the all identities in (g2) and (g3) except the last identity in (g2), which remains unchanged.

\end{proof}

\begin{example}\label{Ex2}
We consider $\dd = (1+ |\j|^2)^{\frac{p'-2}{2}}\j$ with $p'=p/(p-1)$, $p\in (1, \infty)$. This means that
$$
\g(\j,\dd)=(1+ |\j|^2)^{\frac{p'-2}{2}}\j-\dd.
$$
\end{example}
\begin{proof}[Validity of (g1)--(g4) for Example~\ref{Ex2}]
Clearly, $\g(\0,\0)=\0$, $\g_{\dd}(\j,\dd)=-\I$ and
$$
\g_{\j}(\j,\dd)=(1+ |\j|^2)^{\frac{p'-2}{2}} \I + (p'-2)(1+ |\j|^2)^{\frac{p'-4}{2}}\j\otimes\j,
$$
where $(\j\otimes\j)_{k\ell}:=j_k j_\ell$. Hence, for all $\x\in\R^d$, one has
\begin{align*}
\g_{\j}(\j,\dd)\x \cdot \x &= (1+ |\j|^2)^{\frac{p'-4}{2}}\left((1+ |\j|^2) |\x|^2 + (p'-2)(\j\cdot\x)^2\right) \\
&\ge
\begin{cases} (1+ |\j|^2)^{\frac{p'-2}{2}}|\x|^2 & \quad \textrm{ for } p'\ge 2, \\
(p'-1)(1+ |\j|^2)^{\frac{p'-2}{2}}|\x|^2 & \quad \textrm{ for } p'\in (1,2).
\end{cases}
\end{align*}
Hence $\g_{\j}(\j,\dd)>0$. Consequently, $\g_{\j}(\j,\dd) - \g_{\dd}(\j,\dd) >0$ and $\g_{\dd}(\j,\dd) (\g_{\j}(\j,\dd))^T<0$ and the validity of (g1) and (g2) is verified. Furthermore, for any $\dd\in\R^d$, recalling that $p'>1$, we obtain that
$$
\g(\j,\dd) \cdot \j = (1+ |\j|^2)^{\frac{p'-2}{2}} |\j|^2 - \j\cdot\dd \to \infty \quad \textrm{ as } |\j|\to \infty.
$$
Similarly, for any $\j\in\R^d$,
$$
\g(\j,\dd) \cdot \dd = (1+ |\j|^2)^{\frac{p'-2}{2}} \j\cdot\dd  - |\dd|^2 \to - \infty \quad \textrm{ as } |\dd|\to \infty,
$$
and (g3) holds. Finally,
\begin{align*}
  \j\cdot\dd = (1+|\j|^2)^{\frac{p'-2}{2}} |\j|^2 \ge
	\begin{cases} |\j|^{p'-2}|\j|^2 = |\j|^{p'} & \quad\textrm{ if } p'\ge 2, \\
	2^{\frac{p'-2}{2}} |\j|^{p'-2}|\j|^2 = 2^{\frac{p'-2}{2}} |\j|^{p'} & \quad\textrm{ if } p'\in (1,2) \textrm{ and } |\j|\ge 1, \\
	2^{\frac{p'-2}{2}} |\j|^2 \ge 2^{\frac{p'-2}{2}} |\j|^{p'} - c_0 & \quad\textrm{ if } p'\in (1,2) \textrm{ and } |\j| < 1,
	\end{cases}
\end{align*}
where, in the last step, we used Young's inequality $|\j|^{p'} \le |\j|^2 + c$. Since
$$
|\dd|^2 = (1+|\j|^2){^{p'-2}} |\j|^2 \le (1+|\j|^2)^{p'-1} \quad\implies \quad 1+|\j|^2 \ge |\dd|^{\frac{2}{p'-1}},
$$
we observe that
\begin{align*}
\j\cdot\dd &= (1+|\j|^2)^{\frac{p'-2}{2}} |\j|^2 = (1+|\j|^2)^{\frac{p'-2}{2}} (1+ |\j|^2 -1) \\
           &= (1+|\j|^2)^{\frac{p'}{2}} - (1+|\j|^2)^{\frac{p'-2}{2}} \\
					 &\ge \begin{cases}
					\frac12 |\dd|^p - c & \quad \textrm{ if } p'\ge 2,\\
					|\dd|^p - 1 & \quad \textrm{ if } p'\in (1,2)\,,
					\end{cases}
\end{align*}
where in the last step we used Young's inequality $(1+|\j|^2)^{\frac{p'-2}{2}} \le \frac12(1+|\j|^2)^{\frac{p'}{2}} + c$ for $p'\ge 2$ and the fact that
$(1+|\j|^2)^{\frac{p'-2}{2}}\le 1$ for $p'\in (1, 2{)}$. The last two formulae imply (g4).
\end{proof}

\begin{example}\label{Ex3}
For $\dd = (|\j|-\sigma_*)^+\frac{\j}{|\j|}$,  we set\footnote{Here and in what follows we tacitly assume a continuous extension at $\j=\0$, namely $\dd=\0$ for $\j=\0$.}
$$
  \g(\j,\dd) = (|\j|-\sigma_*)^+\frac{\j}{|\j|} -\dd.
$$
\end{example}
\begin{proof}[Validity of (g1)--(g4) for Example~\ref{Ex3}]
Clearly, $\g$ is continuous on $\R^d\times\R^d$, Lipschitz continuous almost everywhere in $\R^{d}\times\R^d$, $\g(\0,\0)=\0$, $\g_{\dd}(\j,\dd) = -\I$ and
$$
  \g_{\j}(\j,\dd) = \frac{(|\j|-\sigma_*)^+}{|\j|}\I + \chi_{\{|\j|> \sigma_*\}} \frac{\j\otimes\j}{|\j|^2} - (|\j|-\sigma_*)^+\frac{\j\otimes\j}{|\j|^3},
$$
where $\chi_{U}$ denotes the characteristic function of $U\subset \R^{d}$. The last identity leads to
$$
  \g_{\j}(\j,\dd)\x\cdot\x \ge \frac{(|\j|-\sigma_*)^+}{|\j|} \left( |\x|^2 - \frac{(\j\cdot\x)^2}{|\j|^2} \right) + \chi_{\{|\j|> \sigma_*\}} \frac{(\j\cdot\x)^2}{|\j|^2} \ge 0.
$$
The above observations imply that $\g_{\j}(\j,\dd)\ge 0$, $\g_{\j}(\j,\dd) - \g_{\dd}(\j,\dd) > 0$ and $\g_{\dd}(\j,\dd)(\g_{\j}(\j,\dd))^T
\ge 0$. Hence, (g1) and (g2) hold.

Next, it is easy to deduce that
\begin{align*}
\g(\j,\dd)\cdot \j &= (|\j| - \sigma_*)^{+}|\j| - \dd\cdot\j \to \infty &&\textrm{ as } |\j|\to \infty, \\
\g(\j,\dd)\cdot\dd &= (|\j| - \sigma_*)^{+}\frac{\j\cdot\dd}{|\j|} - |\dd|^2 \to - \infty  &&\textrm{ as } |\dd|\to \infty
\end{align*}
and consequently, (g3) follows. Finally, we observe that
\begin{align*}
  \j\cdot\dd = (|\j| - \sigma_*)^{+} |\j| \ge
	\begin{cases} 0 \ge |\j|^2 - \sigma_*^2 & \quad\textrm{ if } |\j|\le \sigma_*, \\
	|\j|^2 - \sigma_*|\j| \ge \frac12 |\j|^2 - c & \quad\textrm{ if } |\j| > \sigma_*.
	\end{cases}
\end{align*}
Since $|\dd| = (|\j| - \sigma_*)^{+}$ and consequently $\dd=\0$ if $|\j|\le \sigma_*$ and $|\j| = |\dd|+ \sigma_*$ is $|\j|\ge \sigma_*$, we also get
\begin{align*}
  \j\cdot\dd = (|\j| - \sigma_*)^{+} |\j| \ge
	\begin{cases} 0 = |\dd|^2 & \quad\textrm{ if } |\j|\le \sigma_*, \\
	|\dd|\left( |\dd| + \sigma_*\right) \ge |\dd|^2 - c_0 & \quad\textrm{ if } |\j| > \sigma_*.
	\end{cases}
\end{align*}
This proves (g4).
\end{proof}

{We end this part by studying the models depicted in Figure~\ref{figure1}.}

\begin{example}\label{ExSchod}
For $a:[0,\infty]\to [0,1]$ defined through
\begin{equation}\label{dfsc}
a(x):=\left\{ \begin{aligned}
&1 &&\textrm{for } x\in[0,\sqrt{2}/2],\\
&\frac{\sqrt{2}-x}{x} &&\textrm{for } x\in (\sqrt{2}/2, \sqrt{2}),\\
&0 &&\textrm{for } x\ge \sqrt{2},
\end{aligned}
\right.
\end{equation}
we set
\begin{equation}\label{schod}
\g(\j,\dd):= \j -\dd -a\left(\frac{\sqrt{2}|\j+\dd|}{2}\right)(\j+\dd).
\end{equation}
Then the null points of $\g$ describes the right graph drawn in Figure~\ref{figure1}, i.e.,
\begin{equation}\label{explschod}
|\j|\le 1 \quad \textrm{if } \dd = \0  \qquad\textrm{ and } \qquad \j=\max\{1,|\dd|^{-1}\} \dd \quad \textrm{if } \dd\neq \0.
\end{equation}
In addition, $\g$ satisfies the assumptions (g1)-(g4) with $p=2$.
\end{example}

\begin{proof}[Verification of \eqref{explschod}]
Consider first $\dd=\0$. Then it follows from \eqref{schod} and $\g(\j,\0)=\0$ that
$$
\j =a\left(\frac{\sqrt{2}|\j|}{2}\right)\j
$$
Hence, either $\j=\0$ or
$$
a\left(\frac{\sqrt{2}|\j|}{2}\right)=1.
$$
It however follows from the definition of $a$, see \eqref{dfsc}, that the second option is possible if and only if $|\j|\le 1$.

Next, let $\dd\neq \0$. Then it follows from \eqref{schod} that
\begin{equation}\label{pepa77}
\g(\j,\dd) = \0 ~\Longleftrightarrow~ \left( 1 - a \left(\frac{\sqrt{2}|\j+\dd|}{2}\right)\right)\j = \left( 1+ a \left(\frac{\sqrt{2}|\j+\dd|}{2}\right)\right)\dd,
\end{equation}
and, as $\dd\neq\0$, $\left( 1 - a \left(\frac{\sqrt{2}|\j+\dd|}{2}\right)\right)$ cannot be zero, which means that $|\j+\dd|>1$. Hence, the null points of $\g$ satisfy
$$
\j=b\dd.
$$
The goal is to determine $b$. First, we observe from \eqref{pepa77} and the definition of $a$ that $\j = \dd$ (and thus $b=1$) if $|\j+\dd|\ge 2$.
It remains to show that $b=|\dd|^{-1}$ if $1<|\j+\dd|< 2$. Inserting $\j=b\dd$ into \eqref{schod} we obtain
$$
(b-1)\dd =(1+b)a\left(\frac{\sqrt{2}(1+b)|\dd|}{2}\right)\dd
$$
and consequently
$$
(b-1) =(1+b)a\left(\frac{\sqrt{2}(1+b)|\dd|}{2}\right).
$$
In order to use the fact that $a(x) x = \sqrt{2} - x$ for $x\in (\sqrt{2}/2,\sqrt{2})$, we multiply the last equality by $\tfrac{\sqrt{2}}{2} |\dd|$ and conclude that
$$
   \frac{\sqrt{2}}{2}(b-1) |\dd| = \sqrt{2} - \frac{\sqrt{2}}{2}(1+b) |\dd|,
$$
which gives $b=|\dd|^{-1}$.

\noindent\emph{Validity of (g1)--(g4) for Example~\ref{ExSchod}.}
Obviously, (g1) holds. Next, taking the scalar product of $\g(\j,\dd) = \0$ first by $\dd$ and then by $-\j$ and summing the results, we obtain
\begin{equation*}
\begin{split}
  2\j\cdot\dd & = |\j|^2 + |\dd|^2 - a(\dots) |\j|^2 + a(\dots) |\dd|^2 \\
	&\ge |\j|^2 + |\dd|^2 - a(\dots) |\j|^2 \\
	&\ge \begin{cases} |\j|^2 + |\dd|^2 & \textrm{ if } |\j+\dd| >2, \\
	|\j|^2 + |\dd|^2 - C & \textrm{ if } |\j+\dd| \le 2,\end{cases}
\end{split}
\end{equation*}
which gives (g4) with $p=2$, and also (g3). It remains to show the validity of (g2). Note that
\begin{align*}
\g_{\j}(\j,\dd) &= \left(1-a\left(\frac{\sqrt{2}|\j+\dd|}{2}\right)\right)\I -a'\left(\frac{\sqrt{2}|\j+\dd|}{2}\right)\frac{\sqrt{2}}{2}\frac{(\j+\dd)\otimes (\j+\dd)}{|\j+\dd|} \\
\g_{\dd}(\j,\dd) & = -\left(1+a\left(\frac{\sqrt{2}|\j+\dd|}{2}\right)\right)\I -a'\left(\frac{\sqrt{2}|\j+\dd|}{2}\right)\frac{\sqrt{2}}{2}\frac{(\j+\dd)\otimes (\j+\dd)}{|\j+\dd|}
\end{align*}
Then, by using the definition of $a$, we observe that for arbitrary $\x\in\R^d$
\begin{align*}
\g_{\j}(\j,\dd)\x &\cdot\x=\left\{
\begin{aligned}
&0&&\textrm{for } |\j+\dd|\le 1,\\
&|\x|^2 &&\textrm{for } |\j+\dd|\ge 2,\\
&\left(2-\frac{2}{|\j+\dd|}\right)|\x|^2+2\frac{((\j+\dd)\cdot\x)^2}{|\j+\dd|^3} &&\textrm{for } |\j+\dd|\in (1,2)\\
\end{aligned}
\right. \\
\g_{\dd}(\j,\dd)\x &\cdot\x=\left\{
\begin{aligned}
&-2|\x|^2  &&\textrm{for } |\j+\dd|\le 1,\\
&-|\x|^2 &&\textrm{for } |\j+\dd|\ge 2,\\
&-\frac{2|\x|^2}{|\j+\dd|} +2\frac{({(\j+\dd)\cdot\x})^2}{|\j+\dd|^3} &&\textrm{for } |\j+\dd|\in (1,2)\\
\end{aligned}
\right.
\end{align*}
Thus, $\g_{\j}(\j,\dd)\ge 0$ and $\g_{\dd}(\j,\dd)\le 0$. In addition also $\g_{\j}(\j,\dd)-\g_{\dd}(\j,\dd)>0$. Finally, as $\g_{\j}$ and $\g_{\dd}$ are symmetric, it follows directly from the nonnegativity of $\g_{\j}$ and the nonpositivity of $\g_{\dd}$ that
$$
\g_{\j}(\g_{\dd})^T = \g_{\j}\g_{\dd}\le 0.
$$
Thus (g2) holds.
\end{proof}

\begin{example}\label{Exsmykala}
Let  $b:[0,\infty]\to \mathbb{R}$ be $\sqrt{2}$-periodic and satisfy
$$
b(x):= xa(x) \quad \textrm{for } x\in [0,\sqrt{2}],
$$
where $a$ is defined in~\eqref{dfsc}. Defining
\begin{equation}\label{dfsc2}
\tilde{a}(x):=\left\{
\begin{aligned}
&0  &&\textrm{for } x=0,\\
&\frac{b(x)}{x} &&\textrm{for } x\in (0,\infty),
\end{aligned}
\right.
\end{equation}
we set
\begin{equation}\label{smykala}
\g(\j,\dd):= \j -\dd -\tilde{a}\left(\frac{\sqrt{2}|\j+\dd|}{2}\right)(\j+\dd).
\end{equation}
Then the null points of $\g$ describe the graph drawn on the left in Figure~\ref{figure1}.
In addition, $\g$ satisfies the assumptions (g1)-(g4) with $p=2$.
\end{example}
We do not verify {the} validity of (g1)--(g4) for the graph described by the null points of $\g$ defined in \eqref{smykala} as the proof is almost identical to the proof in Example~\ref{ExSchod}.

\section{The Maxwell--Stefan system}\label{App2}
Here, we consider the Maxwell--Stefan system given by \eqref{m-s1*}. We omit the dependence of parameters on the solution itself and we just focus on the proof of the fulfillment of (G1)--(G3).
\begin{example}\label{Ex5}
For $d,N\in \mathbb{N}$, $N\ge 2$, consider
$$
(\G(\S,\DD))_{\nu i} = \sum_{\mu=1}^{N}\left(\mathbb{A}_{\nu \mu} ({u} _{\mu}\S_{\nu i}-{u} _{\nu}\S_{\mu i})\right) - \DD_{\nu i}, \qquad i=1,\ldots, d; \ \nu=1,\ldots, N,
$$
where $\mathbb{A}$ is a given symmetric  matrix in $\mathbb{R}^{N\times N}$ fulfilling $\mathbb{A}_{\nu\mu}>0$ for $\nu,\mu=1,\ldots, N$ and $\{{u} _{\nu}\}_{\nu=1}^{N}$ fulfill
$$
{u} _{\nu} \in (0,1) \textrm{ for all } \nu=1,\ldots, N \qquad \textrm{ and } \qquad \sum_{\nu=1}^{N} {u} _\nu = 1.
$$
\end{example}

\begin{proof}[Validity of (G1)--(G3)]
We can evaluate
\begin{align*}
\frac{\partial (\G(\S,\DD))_{\nu i}}{\partial \S_{\mu j}}&= \delta_{ij}\left(\delta_{\nu \mu} \left(\sum_{\alpha=1}^{N}\mathbb{A}_{\nu \alpha} {u} _{\alpha}\right) -\mathbb{A}_{\nu \mu} {u} _{\nu}\right),\\
\frac{\partial (\G(\S,\DD))_{\nu i}}{\partial \DD_{\mu j}}&= -\delta_{ij} \delta_{\nu \mu}.
\end{align*}
Then, for an arbitrary $\BB\in \R^{N\times d}$, we have
\begin{align*}
\sum_{\nu,\mu=1}^{N}\sum_{i,j=1}^d\frac{\partial (\G(\S,\DD))_{\nu i}}{\partial \DD_{\mu j}}\BB_{\nu i}\BB_{\mu j}&= -|\BB|^2\le 0,\\
\sum_{\nu,\mu=1}^{N}\sum_{i,j=1}^d\frac{\partial (\G(\S,\DD))_{\nu i}}{\partial \S_{\mu j}}\BB_{\nu i}\BB_{\mu j}&= \sum_{\nu,\mu=1}^{N}\sum_{i=1}^d\BB_{\nu i}\BB_{\nu i}\mathbb{A}_{\nu \mu} {u} _{\mu}-\sum_{\nu,\mu=1}^{N}\sum_{i=1}^d\BB_{\nu i}\BB_{\mu i}\mathbb{A}_{\nu \mu} {u} _{\nu}.
\end{align*}
While the first inequality is exactly of the form we want, we focus on the second inequality. First, we can observe that the second identity can be rewritten in the form
\begin{align*}
\sum_{\nu,\mu=1}^{N}\sum_{i,j=1}^d\frac{\partial (\G(\S,\DD))_{\nu i}}{\partial \S_{\mu j}}\BB_{\nu i}\BB_{\mu j}&= \sum_{i=1}^d\sum_{\nu,\mu=1}^{N}\BB_{\nu i}\BB_{\mu i} \mathbb{B}_{\nu \mu},
\end{align*}
where $\mathbb{B}$ is a  matrix given as
$$
\mathbb{B}_{\nu \mu}:= \left\{\begin{aligned}&\sum_{\alpha=1; \,\alpha\neq \nu}^N\mathbb{A}_{\nu \alpha} {u} _{\alpha} &&\textrm{for }\nu=\mu,\\
&-\mathbb{A}_{\nu\mu}{u} _{\nu} &&\textrm{for }\nu\neq\mu.
\end{aligned} \right.
$$
Next, we can use~\cite[Lemma 2.1]{jungel}, where it is shown that the spectrum of $\mathbb{B}$ is nonnegative (but contains the simple eigenvalue $0$) and consequently it follows that $\frac{\partial (\G(\S,\DD))}{\partial \S}\ge 0$. Hence, we see that $\G$ satisfies (G1) and (G2). Also it is evident that it satisfies (G3)$_2$.  However, since the spectrum of $\mathbb{B}$ also contains $0$ it cannot satisfy (G4). Nevertheless, since for all null points we have that (note that all null points must satisfy $\sum_{\nu} \DD_{\nu}=0$)
$$
\S :\DD = \mathbb{B} \S : \S,
$$
it follows from the positivity of the spectrum of $\mathbb{B}$, except simple eigenvalue zero, that for all $\S$ satisfying $\sum_{\mu} \S_{\mu}=0$ there holds
$$
 \mathbb{B} \S : \S\ge c|\mathbb{B}\S|^2,
$$
and consequently also,
$$
\S :\DD\ge c|\DD|^2.
$$
Hence, (G4) with $p=2$ is fulfilled on the range of $\mathbb{B}$, as also used in the analysis, see \cite{jungel}.

\end{proof}

\section{Solvability of \texorpdfstring{\eqref{problem}}{1.6} for Lipschitz continuous and uniformly monotone graphs}\label{App3}

Here, we consider the following problem: for given $\o\subset \mathbb{R}^d$, $T>0$, $\f:Q\to \R^N$ and $\u{_0}:\o\to\R^N$, find $(\u,\S):Q\to \R^N\times\R^{N\times d}$ satisfying
\begin{subequations}\label{S2problem}
\begin{align}
\pt \u- \diver \S &= \f &&\text{in } Q, \label{S2a}\\
\S &= \S^*(\nabla \u)&&\text{in } Q,\label{S2f}\\
\u&=\0 &&\text{on } \Sigma_D,\label{S2d}\\
\S \n&=\0 &&\text{on } \Sigma_N, \label{S2d2}\\
\u(0, \cdot)&=\u_0  &&\text{in } \o,\label{S2e}
\end{align}
\end{subequations}
where $\S^*:\R^{{N\times d}} \to \R^{{N\times d}}$ is a Lipschitz continuous and uniformly monotone single-valued mapping, which means that there are $C_1$, $C_2>0$ such that, for all $\DD_1$, $\DD_2\in \R^{{N\times d}}$,
\begin{equation}
\begin{split}
|\S^*(\DD_1)-\S^*(\DD_2)|&\le C_2|\DD_1-\DD_2|,\\
(\S^*(\DD_1)-\S^*(\DD_2)): (\DD_1-\DD_2)&\ge C_1|\DD_1-\DD_2|^2,\\
\S^*(\0)&=\0.
\end{split}
\label{defmono}
\end{equation}
Note that taking $\DD_2= \0$ and relabelling $\DD_1$ by $\DD$ in \eqref{defmono} we obtain
\begin{equation}
\S^*(\DD):\DD \ge \frac{C_1}{2}|\DD|^2 + \frac{C_1}{2C_2^2} |\S^*(\DD)|^2 \ge C \left(|\DD|^2 + |\S|^2\right),
\label{pepa_99}
\end{equation}
where we set $\S = \S^*(\DD)$ and $C:=\min\{C_1/2, C_1/(2C_2^2)\}$.
Consequently, the graph $\A$ defined through the relation
\begin{equation}\label{selectA}
(\S, \nabla \u) \in \A ~\Longleftrightarrow~ \S = \S^*(\nabla \u)
\end{equation}
is a Lipschitz continuous and uniformly monotone $2$-coercive graph.

By the Faedo-Galerkin method, we establish the following well-posedness result.
\begin{lemma}\label{S2thm}
Let $\o \subset \R^d$ be a Lipschitz domain, $T>0$, $\f\in L^2(0,T;V^*)$, $\u_0 \in H$ and $\S^*$ satisfy \eqref{defmono}. Then, there exists a unique  couple $(\u, \S)$ such that
\begin{align*}
\u &\in  L^2(0, T; V) \cap  \C([0, T]; H), \\
\pt \u &\in L^2(0, T; V^*),\\
\S &\in L^2(Q; \R^{{N\times d}}),
\end{align*}
satisfying
\begin{subequations}
\begin{align}\label{WFS2}
\langle \pt \u,  \vp \rangle_V+ \io \S : \nabla \vp \d x &= \langle \f,  \vp \rangle_V \qquad \textrm{ for a.a. } t\in (0,T) \textrm{   and for all } \vp\in V, \\ \label{pepa_100}
\S&=\S^*(\nabla \u) \qquad \textrm{ almost everywhere in } Q, \\ \label{ICS2}
\lim_{t\to 0_+} \|\u(t) - \u_0\|_H &= 0.
\end{align}
\end{subequations}
\end{lemma}
\begin{remark}  Obviously, we could completely avoid using $\S$ in the formulation of Lemma \ref{S2thm} and merely require that $\u$ fulfills, instead of \eqref{WFS2}-\eqref{pepa_100},
$$
\langle \pt \u,  \vp \rangle_V+ \io \S^*(\nabla \u) : \nabla \vp \d x = \langle \f,  \vp \rangle_V \qquad \textrm{ for a.a. } t\in (0,T) \textrm{   and for all } \vp\in V.
$$
The formulation used in Lemma \ref{S2thm} is more suitable for proving Theorem \ref{result} in this text.
\end{remark}
\begin{proof} We follow the original Minty method, see \cite{Minty}, with small modifications adapted to our setting.
The whole proof is split into several steps.

\paragraph{{\bf Step 1.} \underline{Galerkin approximations}} Let $\{\w_i\}_{i\in \N}$ and the corresponding $\lambda_i$ be the solutions of the eigenvalue problem $((\w_i, \z)) = \lambda_i (\w_i, \z)$ valid for all $\z\in V$. Here, $((\cdot,\cdot))$ stands for the scalar product in $V$ and $(\cdot,\cdot)$ is the scalar product in $H$, whereas the spaces $V$ and $H$ are defined in Section \ref{Sec-2}. Then $\{\w_i\}_{i\in \N}$ forms an orthogonal basis in $V$ that is in addition orthonormal in~$H$. Furthermore, the projection $P^n$ of $V$ on the linear hull of $\{\w_i\}_{i=1}^n$, defined by
\begin{equation}
P^n \u:= \sum_{i=1}^n (\u, \w_i)_H \w_i,
\end{equation}
satisfies $\|P^n \u\|_{H} \le \|\u\|_{H}$ and $\|P^n \u\|_{V} \le \|\u\|_{V}$. See, for example \cite[Section 6.4]{MNRR96} for details.

For every $n\in \mathbb{N}$, we set 
\begin{equation}\label{ga2}
\u^n(t,\x):= \sum_{i=1}^{n} c^n_i(t) \w_i(\x)~\text{ for }~(t, \x) \in Q,
\end{equation}
where the functions $c^n_i(t)$ solve the following system of ordinary differential equations
\begin{subequations}
\begin{equation}\label{galerkin2}
(\pt \u^n, \w_i)_H + \io \S^*(\D^n) \!: \!\nabla\w_i \d x=\langle \f,  \w_i \rangle_V, \qquad i=1,\dots,n,
\end{equation}
with the initial conditions
\begin{equation}\label{galerkinIC2}
c^n_i( 0)= \io \u_0 \cdot \w_i \d x =(\u_0,\w_i)_H, \qquad i=1,\dots,n.
\end{equation}
\end{subequations}
Thanks to the Picard--Lindel\"{o}f theorem, there exists a unique solution defined on an interval $[0, t^n)$. By virtue of the uniform estimates established in \eqref{odhadv2} below, one observes that $t^n\ge T$ for all~$n$.

\paragraph{{\bf Step 2.} \underline{Uniform estimates}}
Multiplying the $i$-th equation in \eqref{galerkin2} by $c^n_i(t)$ and summing the result over $i=1,\dots,n$, we obtain
\begin{align}\label{testvn2}
\frac12 \dt \|\u^n\|^2_H + \io \S^*(\D^n) : \D^n \d x = \langle \f,  \u^n \rangle_V.
\end{align}
Then, by means of H\"older's and Young's inequalities and \eqref{pepa_99}, followed by integration over~$(0,t)$, we conclude, using also the assumptions on the data $\u_0$ and $\f$, that
\begin{equation}\label{odhadv2}
\begin{aligned}
\sup_{t \in (0,T)} \|\u^n(t)\|^2_{H} &+ \iT \|\u^n\|^2_V + \|\S^*(\nabla \u^n)\|^2_{L^2(\o{;\R^{N\times d}})} \d t \leq C \left( \iT \| \f \|^2_{V^*} \d t + \|\u_0\|^2_H \right) \leq C.
\end{aligned}
\end{equation}
This implies the following $n$-independent estimate
\begin{equation}\label{uniformvn2}
\|\u^n\|_{L^2(0,T;V) \cap L^\infty(0,T;H)} + \|\S^*(\nabla \u^n)\|_{L^2(Q{;\R^{N\times d}})} \leq C.
\end{equation}
Furthermore, for any $\vp \in V$, we obtain from \eqref{galerkin2}
\begin{equation*}
\langle \pt \u^n, \vp \rangle_V = (\pt \u^n, P^n\vp)_H = \io \pt \u^n \cdot (P^n\vp) \d x
=- \io \S^*(\D^n) : \nabla (P^n\vp) \d x + \langle \f, P^n\vp \rangle_V.
\end{equation*}
Standard duality and scalar product estimates together with \eqref{uniformvn2} and the continuity of~$P^n$ mentioned above imply that
\begin{equation}\label{uniformptvn2}
\iT \|\pt \u^n\|^2_{V^*} \d t 
\leq C, \qquad\qquad~\text{ uniformly with respect to } n \in \N.
\end{equation}


\paragraph{{\bf Step 3.} \underline{Limit \texorpdfstring{$n\to \infty$}{n}}} 

By virtue of the uniform estimates \eqref{uniformvn2} and  \eqref{uniformptvn2}, the reflexivity of spaces $V$ and $V^*$ and the Aubin--Lions lemma, there exist (not relabelled) subsequences and functions $\u$ and $\S$ such that, for $n \to \infty$,
\begin{subequations}\label{convergences2}
\begin{align}
\u^n &\tow^* \u &&\text{weakly$^*$ in } L^\infty(0, T; H),\\
\u^n &\tow \u &&\text{weakly in } L^2(0, T; V), \\
\pt \u^n &\tow \pt \u &&\text{weakly in } L^2(0, T; V^*), \\
\u^n &\to \u &&\text{strongly in } L^2(0, T; H), \label{convergence-strong2}\\
\S^*(\D^n) &\tow \S &&\text{weakly in } L^2(Q; \R^{{N\times d}}).
\end{align}
\end{subequations}
For any $\xi \in C^1(0, T)$ and $\vp \in V$,  multiplying the $i$-th equation by $\xi(\vp, \w_i)_H$, summing the result over $i=1,\ldots, k$ for $k\leq n$ and integrating then the outcome over $(0,T)$, we get, for every $k=1,\ldots, n${,}
\begin{align*}
\iT \left(\pt \u^n, \xi P^k\vp \right)_H \d t + \iq \S^*(\nabla \u^n) : \nabla (P^k\vp)\xi  \d x \d t = \iT \langle \f, \xi P^k\vp \rangle_V \d t.
\end{align*}
Using the convergence results \eqref{convergences2}, we can easily take the limit for $n \to \infty$. Since the limit terms hold for any smooth $\xi$, we obtain
\begin{equation*}
\langle \pt \u, P^k\vp \rangle_V + \io \S: \nabla(P^k\vp) \d x = \langle \f, P^k\vp \rangle_V \qquad \textrm{ for a.a. } t\in (0,T) \textrm{ and for all } k \in \N.
\end{equation*}
As $P^k \vp \to \vp$ in $V$ as $k \to \infty$, we arrive at the weak formulation \eqref{WFS2}.

\paragraph{{\bf Step 4.} \underline{Attainment of the initial datum}}


We first notice that it follows from $\u \in L^2(0,T; V)$ and $\pt \u \in L^2(0,T; V^*)$ that $\u \in \C([0,T]; H)$. Hence
\begin{equation}\label{strongconvpocpod2}
\u(t) \to \u(0) \qquad \text{ strongly in } H \text{ as } t \to 0_+.
\end{equation}
To prove \eqref{ICS2}, it is then enough to show that
\begin{equation}
\u(t) \tow \u_0 \quad \textrm{ weakly in } H \textrm{ as } t \to 0_+. \label{pepa_98}
\end{equation}
To this end, let $0< \e \ll 1$ and $t \in (0,T-\e)$. Recalling the definition of an auxiliary $\eta$ in \eqref{eta}, multiplying \eqref{galerkin2} by such an $\eta$ and integrating the result with respect to $\tau \in (0,T)$, we obtain, for every $i = 1, \ldots, n$,
\begin{equation*}
\iT (\pt \u^n, \w_i)_H \eta \d \tau +  \iq \S^*(\nabla \u^n) \!:\! \nabla \w_i \eta\d x \d \tau = \iT \langle \f, \w_i\rangle_V \eta \d \tau.
\end{equation*}
Integration by parts in the first term (using $\eta(T)=0$) then leads to
\begin{align*}
-\iT (\u^n, \w_i)_H \eta' \d \tau + \iq \S^*(\nabla \u^n) : \nabla \w_i \eta\d x \d \tau = \iT \langle \f, \w_i\rangle_V \eta \d \tau+ (P^n \u_0, \w_i)_H \eta(0).
\end{align*}
Applying the weak convergence results established in \eqref{convergences2} as well as the convergence of the projection $P^n$
as $n \to \infty$ we observe, for any $i \in \N$, that
\begin{align*}
-\iT (\u, \w_i)_H \eta' \d \tau +  \iq \S \!:\! \nabla\w_i \eta\d x \d \tau = \iT \langle \f, \w_i\rangle_V \eta \d \tau+ (\u_0, \w_i)_H \eta(0).
\end{align*}
By noting the properties of $\eta$, namely that $\eta(\tau) = 1$ for $\tau\in [0,t)$, $\eta(\tau) = 0$ for $\tau\in (t + \e, T]$, and $\eta'(\tau) = -\frac{1}{\e}$ for $\tau \in (t,t + \e)$, then yields
\begin{align*}
\frac{1}{\e}\int_t^{t+\e} (\u, \w_i)_H \d \tau +  \int_{Q_{t+\e}} \S :\nabla \w_i \eta \d x \d \tau = \int_0^{t+\e} \langle \f, \w_i\rangle_V \eta\d \tau + (\u_0, \w_i)_H.
\end{align*}
Finally, we let $\e \to 0_+$. In the first term, the integrand is well-defined (in fact, $\u \in \C([0,T];H)$), and the term converges to $(\u(t), \w_i)_H$. In the other terms, due to their integrability, we can take the limit as $\e \to 0_+$ together with $t \to 0_+$ and arrive at
\begin{align*}
\lim_{t \to 0_+}(\u(t), \w_i)_H = (\u_0, \w_i)_H.
\end{align*}
Since $\{\w_i\}_{i\in \N}$ forms a basis in $H$, \eqref{pepa_98} and then also \eqref{ICS2} are proved.

\paragraph{{\bf Step 5.} \underline{Attainment of the constitutive equation}}

It remains to show~{\eqref{pepa_100}}. 
To do so, we multiply \eqref{testvn2} by piecewise linear $\eta(\tau)$ defined in \eqref{eta} and integrate the result over $(0,T)$. This yields
\begin{align*}
\int_{Q_{t + \e}} &\S^*(\D^n) : \D^n \eta \d x \d \tau = \int_0^{t+\e}\langle \f, \u^n \rangle_V \eta \d \tau + \frac12 \|P^n \u_0\|_H^2 - \frac{1}{2\e}\int_t^{t + \e} (\u^n, \u^n)_H \d \tau.
\end{align*}
Since $\S^*(\0) = \0$ and $\S^*(\cdot)$ is monotone, we have, for every $n \in \N$,
\begin{align*}
\S^*(\D^n) : \D^n &\geq 0.
\end{align*}
Therefore, as $\eta \equiv 1$ in $Q_t$,
\begin{align*}
\limsup_{n \to \infty} &\int_{Q_{t}} \S^*(\D^n) : \D^n \d x \d \tau \leq \limsup_{n \to \infty} \int_{Q_{t + \e}} \S^*(\D^n) : \D^n \eta \d x \d \tau \\
&= \limsup_{n \to \infty} \int_0^{t+\e}\langle \f, \u^n \rangle_V \eta \d \tau + \frac12 \|P^n \u_0\|_H^2 - \liminf_{n \to \infty} \frac{1}{2\e}\int_t^{t + \e} (\u^n, \u^n)_H \d \tau \\
&\leq \int_0^{t+\e}\langle \f, \u \rangle_V \eta \d \tau + \frac12 \|\u_0\|_H^2 - \frac{1}{2\e}\int_t^{t + \e} (\u, \u)_H \d \tau,
\end{align*}
where we used the results established in \eqref{convergences2} and the weak lower-semicontinuity of the norm. Letting $\e \to 0_+$, we note that the left-hand side is independent of $\e$ and {that} all quantities on the right-hand side are well-defined for the limit (since $\u \in \C([0,T]; H)$). We thus obtain, for an arbitrary $t\in(0,T)$,
\begin{equation}\label{limitujuce}
\begin{aligned}
\limsup_{n \to \infty} &\int_{Q_{t}} \S^*(\D^n) : \D^n \d x \d \tau \leq \it\langle \f, \u \rangle_V \d \tau + \frac12 \left( \|\u_0\|_H^2 - \|\u(t)\|_H^2 \right).
\end{aligned}
\end{equation}
Now, we set $\vp := \u$ in \eqref{WFS2} and integrate the result over the time interval $(0,t)$. Using integration by parts formulae (thanks to the fact that we have the Gelfand triple) and \eqref{ICS2}, we get
\begin{equation}\label{limitne}
\begin{aligned}
\int_{Q_{t}} \S^*(\D) : \D \d x \d \tau &= \it\langle \f, \u \rangle_V - \langle \pt \u, \u \rangle_V \d \tau \\
&= \it\langle \f, \u \rangle_V \d \tau + \frac12 \left( \|\u_0\|_H^2 - \|\u(t)\|_H^2 \right).
\end{aligned}
\end{equation}
Hence, by comparing \eqref{limitujuce} and \eqref{limitne}, we obtain
\begin{equation}\label{odhadbox}
\begin{aligned}
\limsup_{n \to \infty} &\int_{Q_{t}} \S^*(\D^n) : \D^n \d x \d \tau\leq \int_{Q_{t}} \S^*(\D) : \D \d x \d \tau.
\end{aligned}
\end{equation}
Now, let $\W \in L^2(0,T;L^2(\o))$ be arbitrary, then
\begin{align*}
0&\leq \iqt (\S^*(\D^n)-\S^*(\W)):(\D^n-\W) \d x \d \tau \\
&=\iqt \S^*(\D^n):\D^n  \d x \d \tau -\iqt \S^*(\D^n):\W + \S^*(\W):(\D^n-\W) \d x \d \tau.
\end{align*}
Letting $n\to \infty$, using the estimate \eqref{odhadbox} and the weak convergence results stated in~\eqref{convergences2}, we obtain
\begin{equation*}
0\leq \iqt (\S^*(\D)-\S^*(\W)):(\D-\W) \d x \d \tau.
\end{equation*}
Finally, by setting in particularly $\W:= \D \pm \e \Z$, dividing the result by $\e$ and letting $\e \to 0_+$ (at this point we use the continuity of the Lipschitz continuous single-valued mapping $\S^*$), we obtain, for arbitrary~$\Z$,
\begin{equation}\label{minty}
0\leq \iqt (\S-\S^*(\D)):\Z \d x \d \tau,
\end{equation}
which implies that $\S = \S^*(\D)$ in $Q_t$ for any $t\in (0,T)$.

\paragraph{{\bf Step 6.} \underline{{Uniqueness}}} Let $(\u_1,\S_1)$ and $(\u_2,\S_2)$ be two different weak solutions to \eqref{S2problem} corresponding to the same set of data. Subtracting their weak formulations and inserting the defining expressions for $\S_1$ and~$\S_2$, we obtain
\begin{equation*}
\langle \pt (\u_1 - \u_2),  \vp \rangle_V+ \io (\S^*(\D_1) - \S^*(\D_2)) : \nabla\vp \d x = 0 \qquad \textrm{ for all } \vp \in V.
\end{equation*}
Taking $\vp:=(\u_1(t,\cdot) -\u_2(t,\cdot))$, we get
\begin{align*}
\frac{1}{2} \dt \|\u_1 - \u_2\|^2_H &+ \io (\S^*(\D_1) - \S^*(\D_2)) : (\D_1-\D_2) \d x =0,
\end{align*}
which, due to the uniform monotonicity of $\S^*= \S^*(\nabla \u)$, after integration over $(0,t)$ for an arbitrary $t\in(0,T)$, leads to
\begin{equation*}
\|\u_1(t) - \u_2(t)\|^2_H \leq \|\u_1(0) - \u_2(0)\|^2_H=0.
\end{equation*}
Necessarily, $\u_1(t) = \u_2(t)$ in $V$ for almost every $t \in (0,T)$, and obviously, $\S_1=\S^*(\D_1)=\S^*(\D_2)=\S_2$.
\end{proof}

\end{appendix}


\providecommand{\bysame}{\leavevmode\hbox to3em{\hrulefill}\thinspace}
\providecommand{\MR}{\relax\ifhmode\unskip\space\fi MR }
\providecommand{\MRhref}[2]{%
  \href{http://www.ams.org/mathscinet-getitem?mr=#1}{#2}
}
\providecommand{\href}[2]{#2}

\end{document}